\documentclass[10pt,a4paper]{article}

\usepackage[pdftex]{graphicx}
\usepackage{amsmath}
\usepackage{amssymb}
\usepackage{amsthm}
\usepackage{amsfonts}
\usepackage{wrapfig}
\usepackage{calc}
\usepackage{tabulary}
\usepackage{fullpage}
\usepackage{xcolor}
\usepackage{tikz}
\usetikzlibrary{snakes}

\newcommand{\aut}{\textnormal{Aut}}

\newcommand{\supp}{\textnormal{supp}}

\newcommand{\sym}{\textnormal{Sym}(\omega)}
\newcommand{\path}[1]{\textnormal{Path}\langle {#1} \rangle}
\newcommand{\dinfinity}{D_\infty}
\newcommand{\cfpo}[1] {\mathrm{CFPO}_{#1}}
\newcommand{\alt}[1] {\mathrm{Alt}_{#1}}
\newcommand{\treecauser}{dendromorphic group }
\newcommand{\act}{\textnormal{Act}}

\makeatletter
\providecommand*{\cupdot}{%
  \mathbin{%
    \mathpalette\@cupdot{}%
  }%
}
\newcommand*{\@cupdot}[2]{%
  \ooalign{%
    $\m@th#1\cup$\cr
    \sbox0{$#1\cup$}%
    \dimen@=\ht0 %
    \sbox0{$\m@th#1\cdot$}%
    \advance\dimen@ by -\ht0 %
    \dimen@=.5\dimen@
    \hidewidth\raise\dimen@\box0\hidewidth
  }%
}

\providecommand*{\bigcupdot}{%
  \mathop{%
    \vphantom{\bigcup}%
    \mathpalette\@bigcupdot{}%
  }%
}
\newcommand*{\@bigcupdot}[2]{%
  \ooalign{%
    $\m@th#1\bigcup$\cr
    \sbox0{$#1\bigcup$}%
    \dimen@=\ht0 %
    \advance\dimen@ by -\dp0 %
    \sbox0{\scalebox{2}{$\m@th#1\cdot$}}%
    \advance\dimen@ by -\ht0 %
    \dimen@=.5\dimen@
    \hidewidth\raise\dimen@\box0\hidewidth
  }%
}
\makeatother

\newtheorem{lemma}{Lemma}[section]
\newtheorem{theorem}[lemma]{Theorem}
\newtheorem{cor}[lemma]{Corollary}

\newtheorem{dfn}[lemma]{Definition}
\newtheorem{prop}[lemma]{Proposition}

\newtheorem{remark}[lemma]{Remark}

\title{The Reconstruction of Cycle-free Partial Orders from their Abstract Automorphism Groups I : Treelike CFPOs}

\author{Robert Barham \\ Institut f\"ur Algebra, TU Dresden \\ robert.barham@yahoo.co.uk}

\begin{document}

\maketitle

\thanks{The author has received funding from the European Research Council under the European Community's Seventh Framework Programme (FP7/2007-2013 Grant Agreement no. 257039).}

\abstract{In this triple of papers, we examine when two cycle-free partial orders can share an abstract automorphism group.  This question was posed by M. Rubin in his memoir concerning the reconstruction of trees.

In this first paper, we give a variety of conditions that guarantee when a CFPO shares an automorphism group with a tree.  Some of these conditions are conditions on the abstract automorphism group, while some are one the CFPO itself.  Some of the lemmas used have corollaries concerning the model theoretic properties of a CFPO.}

\section{Introduction}\label{intro1}

The question of how much of a given structure is encoded in its symmetries is one that surfaces in many different ways in many different areas of pure mathematics.  One way in which this question surfaces is reconstruction from automorphism groups of first order structures.  An account of the history of this can be found in \cite{Pech2Monoid}, by C. Pech and M. Pech.  There are many levels of structure can be placed on a automorphism group, so here is some notation that clarifies what exactly we mean by ``isomorphism''.

\begin{dfn}
Let $M$ be an structure and let $\mathcal{L}_G$ be the language of groups, $\tau_M$ the topology of pointwise convergence on $\aut(M)$ and $\mathrm{Op}(f,x)$ the group action of $\aut(M)$ on $M$.
$$\begin{array}{rcl}
\aut(M) \cong_A \aut(N) & \Leftrightarrow & \langle \aut(M), \mathcal{L}_G \rangle \cong \langle \aut(N), \mathcal{L}_G \rangle \\
\aut(M) \cong_T \aut(N) & \Leftrightarrow & \langle \aut(M), \mathcal{L}_G, \tau_M \rangle \cong \langle \aut(N), \mathcal{L}_G, \tau_N \rangle \\
\aut(M) \cong_P \aut(N) & \Leftrightarrow & \langle \aut(M), M , \mathcal{L}_G, \mathrm{Op} \rangle \cong \langle \aut(N), N , \mathcal{L}_G, \mathrm{Op} \rangle \\
\end{array}
$$
The subscript $A$ stands for `abstract', $T$ for `topology' and $P$ for `permutation'.
\end{dfn}

One way in which we can pursue the reconstruction of first order models is the search for `faithful' classes.

\begin{dfn}
A class of first order models $K$ is said to be faithful if for all $M,N \in K$
$$ \aut(M) \cong_A \aut(N) \Leftarrow M \cong N$$
\end{dfn}

In this trio of papers, the structures we consider are `cycle-free partial orders' (CFPOs), and our notion of symmetry is the associated automorphism groups in the language of group theory.  CFPOs are a generalisation of trees, or semi-linear orders.  They were introduced in the memoir entitled `The Reconstruction of Trees from Their Automorphism Groups' \cite{Rubin91}, by M. Rubin, as a family of structures where the methods he used for reconstructing the trees would extend.  Their definition can be found in the `Preliminary Definitions' section.

Notable works concering CFPOs include an extensive study of the transitivity properties of the CFPOs can be found in \cite{Warren1997}, a memoir of Richard Warren, who defines them using the notion of path.  Warren's study was extended by two papers in 1998, \cite{Truss1998CFPO} by Truss and \cite{CreedTrussWarren} by Creed, Truss and Warren, both of which add to cases not fully dealt with by Warren in \cite{Warren1997}.  Gray and Truss in \cite{TrussGray09} examine the relationship between ends of a graph and CFPOs, and extract a number of results from this relationship.  This viewpoint is rather illuminating, even if one is not familiar with ends of graphs.

These papers do not extend Rubin's method for reconstruction to the CFPOs, instead applying different methods from a variety of sources.  The first part shows when we can appeal to his results directly.  The second adapts methods used by Shelah in \cite{ShelahPermutation, ShelahPermutationErrata} and by Shelah and Truss in \cite{ShelahTrussQuotients}, while the third uses properties of the wreath products of groups to get a class where the first and second approaches may be used together.

Part I is the draws most strongly from \cite{Rubin91}, the memoir of Rubin.  In that memoir, Rubin gives an extremely complete reconstruction result for trees.  This part seeks to show when we can appeal to those results directly.  Section \ref{intro1} contains the preliminary notions required, as well as the definition of CFPOs.  Section \ref{prelim1} will give the definition of connecting set and path, which will be used extensively throughout this whole work.  Their most immediate use will be to define the class of CFPOs.

Section \ref{orderconditions} describes three constructions that build from a CFPO with some constraints a tree with the same permutation automorphism group.  These constrains are: possessing a fixed point; not embedding $\alt{\omega}$; and not embedding $\alt{}$.  These constructions allow us to deduce that the theory of a CFPO is dp-minimal, which will be done in Section \ref{CFPOsMT}.

Section \ref{gpconditions} shows that if $\dinfinity$, the infinite dihedral group, is a subgroup of an automorphism group of a tree, then it is contained in a supergroup isomorphic to one of a family of groups called `dendromorphic'.  We use this to formulate a condition the abstract automorphism group of a CFPO that shows when it is also the automorphism group of a tree.

These papers are based on the author's Ph.D. thesis, supervised by Prof. John Truss at The University of Leeds.

\section{Preliminaries}\label{prelim1}

This section contains the definition of CFPOs and the supporting concepts.

\begin{dfn}[2.3.2 of \cite{Warren1997}]\label{dfn:connectingset}
If $M$ is a partial order and $a,b \in M$, then $C$, the n-tuple $\langle c_1,c_2, \ldots ,c_n \rangle$ (for $n \geq 2$) is said to be a \textbf{connecting set} from $a$ to $b$ in $M$, written $C \in C^M \langle a, b \rangle$, if the following hold:
\begin{enumerate}
\item $c_1=a, c_n=b, c_2, \ldots ,c_{n-1} \in M^D$
\item if $1 \leq i \leq n-1$, then $c_i \not\parallel c_{i+1}$
\item if $1<i<n$, then $c_{i-1} < c_i > c_{i+1}$ or $c_{i-1} > c_i < c_{i+1}$
\end{enumerate}
\end{dfn}

\begin{dfn}[2.3.3 of \cite{Warren1997}]
Let $M$ be a partial order, $a,b \in M$, and let $C = \langle c_1,c_2,\ldots ,c_n \rangle$ be a connecting set from $a$ to $b$ in $M$.  Let $\sigma_k$ (for $ 1< k < n$) be maximal chains in $M^D$ with endpoints $c_k,c_{k+1} \in \sigma_k$, such that if $x \in \sigma_i \cap \sigma_j$ for some $i < j$, then $j = i+1$ and $x = c_{i+1}$.  Then we say that $P = \bigcup_{0<k<n} \sigma_k$ is a \textbf{path} from $a$ to $b$ in $M$.
\end{dfn}

\begin{dfn}
A partial order $M$ is said to be a \textbf{cycle-free partial order} (CFPO) if for all $x,y \in M$ there is at most one path between $x$ and $y$ in $M^D$. 
\end{dfn}

\begin{dfn}
Let $M$ be a CFPO, and let $x,y \in M$ and $A,B \subseteq M$.  The unique path between $x,y$ is denoted by $\path{x,y}$.  We also define:
$$
\begin{array}{r c l}
\path{x,B} & := & \bigcap_{b \in B} \path{x,b} \\
\path{A,y} & := & \bigcap_{a \in A} \path{a,y} \\
\path{A,B} & := & \bigcap_{a \in A} \path{a,B}
\end{array}
$$
\end{dfn}

\begin{dfn}
A partial order is said to be \textbf{connected} if there is a path between any two points, i.e. $\path{x,y}$ exists for all $x,y \in M$, and is said to be disconnected otherwise.

Let $M$ be a partial order and let $C \subseteq M$.  We say that $C$ is a \textbf{connected component} of $M$ if it is a maximal connected subset of $M$, i.e. for all $x,y \in M$ if $x \in C$ and $\path{x,y}$ exists then $y \in C$.
\end{dfn}

Truss in \cite{Truss01} shows that the class of cycle-free partial orders is axiomatisable, but not finitely axiomatisable.  Adding colour predicates to the language of partial orders is an inalienable part of Rubin's memoir, as well as this work, so throughout these three papers every partial order discussed may be coloured by infinitely many colour predicates.

\begin{dfn}\label{dfn:treelike}
CFPO $M$ is said to be \textbf{treelike} if there is a coloured tree $T$ such that
$$\aut(M) \cong_A \aut(T)$$

If $G \leq \aut(M)$ then the action of $G$ is said to be treelike if there is a tree $T$ such that
$$G \cong_A \aut(T)$$
\end{dfn}

\begin{prop}\label{Lemma:NeededInChapter3}
Let $M$ be a CFPO and let $(a_1, \ldots, a_n),(b_1, \ldots, b_n) \in M$.  Furthermore, we define
$$\mathrm{Adj}:= \lbrace (i,j) \: : \: a_i \leq \geq a_j \textnormal{ and } \forall k \: a_k \not\in ( a_i, a_k) \rbrace$$
Then $(a_1, \ldots a_n)$ and $(b_1, \ldots, b_n)$ lie in the same orbit if and only if there is an isomorphism of finite structures
$$\phi: (a_1, \ldots a_n) \rightarrow (b_1, \ldots, b_n)$$
such that for all $(i,j) \in \mathrm{Adj}$, the pair $(a_i,a_j)$ lies in the same 2-orbit as $(b_{\phi(i)},b_{\phi(j)})$.
\end{prop}
\begin{proof}
This is a quick consequence of Proposition 4.5 of \cite{Simon2011}.
\end{proof}

\begin{dfn}
$\alt{}$ is the partial order with the domain $ \lbrace a_i \; : \; i \in \mathbb{Z} \rbrace $ ordered by
\begin{itemize}
\item if $i$ is odd then $ a_{i-1} > a_{i} < a_{i+1} $ 
\item if $i$ is even then $ a_{i-1} < a_{i} > a_{i+1} $
\end{itemize}
$\alt{n}$ is defined to be $\alt{}$ restricted to $\lbrace a_0, \ldots a_{n-1} \rbrace$.  Note that flipping the order does not affect the definition of $\alt{}$, but does affect $\alt{n}$.  We will write $\alt{n}^*$ for the reverse ordering of $\alt{n}$.

$\alt{\omega}$ is defined to be $\alt{}$ restricted to $\lbrace a_i \: : \: i \in \omega \rbrace$.  Again, the reverse ordering is denoted by $\alt{\omega}^*$
\end{dfn}

\begin{figure}[!ht]
\begin{center}
\begin{tikzpicture}[scale=0.14]
\draw (-25,0) -- (-20,5) -- (-10,-5) -- (0,5) -- (10,-5) -- (20,5) -- (25,0);

\fill (-20,5) circle (0.5);
\fill (-10,-5) circle (0.5);
\fill (0,5) circle (0.5);
\fill (20,5) circle (0.5);
\fill (10,-5) circle (0.5);

\draw[anchor=west] (-20,5) node {$a_{-2}$};
\draw[anchor=west] (-10,-5) node {$a_{-1}$};
\draw[anchor=west] (0,5) node {$a_{0}$};
\draw[anchor=west] (10,-5) node {$a_{1}$};
\draw[anchor=west] (20,5) node {$a_{2}$};

\draw (-27,0) node {$\ldots$};
\draw (27,0) node {$\ldots$};

\end{tikzpicture}
\end{center}
\caption{The Alternating Chain}
\end{figure}
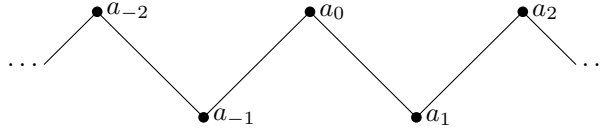

\begin{dfn}
A CFPO is said to be a $\cfpo{n}$ if $\alt{n}$ embeds but $\alt{n+1}$ does not.  A CFPO is said to be a $\cfpo{\omega}$ if $\alt{\omega}$ embeds but $\alt{}$ does not.  A CFPO is said to be a $\cfpo{\infty}$ if $\alt{}$ embeds.
\end{dfn}

\begin{dfn}\label{dfn:G(t)}
If $f \in \aut(M)$ then the \textbf{support} of $f$ is the following set:
$$\supp(f) := \lbrace x \in M \: : \: f(x) \not= x \rbrace$$
If $F \subseteq \aut(M)$ and $x \in M$ then
$$F(x) := \lbrace f(x) \: : \: f \in F \rbrace$$
and the support of $F$ is the following set:
$$\supp(F) := \bigcup_{f \in F} \supp(f)`$$
\end{dfn}

\section{Order Conditions}\label{orderconditions}

We start with CFPOs which have points which are fixed by every automorphism (which we call \textbf{fixed points}).  We will take from the midst of $M$ our fixed point and plant it in the ground, before straightening out the paths of $M$ into branches.

The colouring of $M$ is largely irrelevant for this work, and so takes a very back-seat role.  Indeed, for the rest of this subsection the term `monochromatic' will mean `monochromatic with respect to $U$', where $U$ is the predicate introduced in the next definition.

\begin{dfn}\label{Def:XandY}
Let $\langle M, \leq_M \rangle$ be a connected CFPO whose automorphism group fixes the point $r$.  We will construct $T(M)$ by specifying a new order on $|M |$.  Let $r$ be the fixed point of $M$, which will become the root of $T(M)$.  The colour of $r \in M$ is the same in $T(M)$.

We denote the order on $T$ by $\leq_T$ and define it as follows:
\begin{itemize}
\item $r \leq_{T(M)} s$ for all $s \in M$
\item $s \leq_{T(M)} t$ if and only if $s \in \path{r,t}$
\end{itemize}
We also add a new unary predicate, which we call $U$.  We define the following sets:
$$
\begin{array}{l c l}
X_0 & := & \lbrace t \in M \; : \; r \leq_M t \rbrace \\
Y_0 & := & \lbrace t \in M \; : \; t <_M r \rbrace \\
 & \vdots & \\
X_n & := & \lbrace t \in M \; : \; y \leq_M t \: \textnormal{for some} \: y \in Y_{n-1} \rbrace \setminus \bigcup_{i < n}(X_i \cup Y_i) \\
Y_n & := & \lbrace t \in M \; : \; t <_M x \: \textnormal{for some} \: x \in X_{n-1} \rbrace \setminus \bigcup_{i < n}(X_i \cup Y_i) \\
 & \vdots & \\
\end{array}
$$
We also define $X := \bigcup X_i$ and say that $U(t)$ holds whenever $t \in X$.  Finally

$$\mathcal{X} := \lbrace X_i , Y_i \: : \: i \in \omega \rbrace$$
\end{dfn}

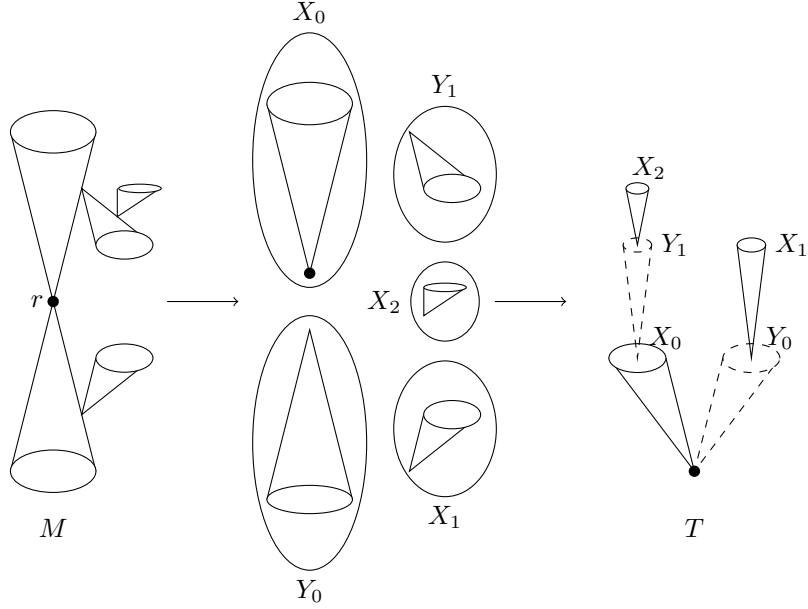
\begin{figure}
\begin{center}
\begin{tikzpicture}[scale=0.075]
\draw (0,30) -- (15,-30);
\draw (0,-30) -- (15,30);
\fill[white] (7.5,30) ellipse  (7.5 and 3.75);
\fill[white] (7.5,-30) ellipse  (7.5 and  3.75);
\draw (7.5,30) ellipse  (7.5 and 3.75);
\draw (7.5,-30) ellipse  (7.5 and 3.75);

\draw (15 ,10) -- (12.5 ,20) -- (25 ,10);
\fill[white] (20,10) ellipse (5 and 2.5);
\draw (20,10) ellipse (5 and 2.5);

\draw (15 ,-10) -- (12.5 ,-20) -- (25 ,-10);
\fill[white] (20,-10) ellipse (5 and 2.5);
\draw (20,-10) ellipse (5 and 2.5);

\draw (18.75 ,20) -- (18.75 ,15) -- (26.25 ,20);
\fill[white] (22.725,20) ellipse (3.725 and 0.75);
\draw (22.725,20) ellipse (3.725 and 0.75);

\fill (7.5,0) circle (1);

\draw[anchor=east] (7.5,0) node {$r$};

\draw (7.5,-40) node {$M$};

\draw[->]  (27.5,0) -- (40,0);

\draw (45,35) -- (52.5,5) -- (60,35);
\fill[white] (52.5,35) ellipse (7.5 and 3.75);
\draw (52.5,35) ellipse (7.5 and 3.75);
\draw (52.5,25) ellipse (10 and 22.5);
\fill (52.5,5) circle (1);

\draw (45,-35) -- (52.5,-5) -- (60,-35);
\fill[white] (52.5,-35) ellipse (7.5 and 2.5);
\draw (52.5,-35) ellipse (7.5 and 2.5);
\draw (52.5,-25) ellipse (10 and 22.5);

\draw (72.5,20) -- (70,30) -- (82.5,20);
\fill[white] (77.5,20) ellipse (5 and 2.5);
\draw (77.5,20) ellipse (5 and 2.5);
\draw (76.25,22.5) ellipse (9 and 12);

\draw (72.5,-20) -- (70,-30) -- (82.5,-20);
\fill[white] (77.5,-20) ellipse (5 and 2.5);
\draw (77.5,-20) ellipse (5 and 2.5);
\draw (76.25,-22.5) ellipse (9 and 12);

\draw (72.525 ,2.5) -- (72.525 ,-2.5) -- (79.975 ,2.5);
\fill[white] (76.25,2.5) ellipse (3.725 and 0.75);
\draw (76.25,2.5) ellipse (3.725 and 0.75);
\draw (76.25,0) ellipse (6 and 7);

\draw[anchor=south]  (52.5,47.5) node {$X_0$};
\draw[anchor=north]  (52.5,-47.5) node {$Y_0$};
\draw[anchor=south] (76.25,34.5) node {$Y_1$};
\draw[anchor=north]  (76.25,-34.5) node {$X_1$};
\draw[anchor=east]  (70.25,0)  node {$X_2$};

\draw[->]  (85,0) -- (97.5,0);

\draw (105,-10) -- (120,-30) -- (115,-10);
\fill[white] (110,-10) ellipse (5 and 2.55);
\draw (110,-10) ellipse (5 and 2.55);

\draw[dashed] (125,-10) -- (120,-30) -- (135,-10);
\fill[white] (130,-10) ellipse (5 and 2.55);
\draw[dashed] (130,-10) ellipse (5 and 2.55);

\draw[dashed] (107.5,10) -- (110,-10) -- (112.5,10);
\fill[white] (110,10) ellipse (2.5 and 1.25);
\draw[dashed] (110,10) ellipse (2.5 and 1.25);

\draw (127.5,10) -- (130,-10) -- (132.5,10);
\fill[white] (130,10) ellipse (2.5 and 1.25);
\draw (130,10) ellipse (2.5 and 1.25);

\draw (108,20) -- (110,10) -- (112,20);
\fill[white] (110,20) ellipse (2 and 1);
\draw (110,20) ellipse (2 and 1);

\fill (120,-30) circle (1);

\draw[anchor=south]  (115,-10) node {$X_0$};
\draw[anchor=south]  (135,-10) node {$Y_0$};
\draw[anchor=west] (112.5,10) node {$Y_1$};
\draw[anchor=west]  (132.5,10) node {$X_1$};
\draw[anchor=south]  (112,20)  node {$X_2$};

\draw (120,-40) node {$T$};
\end{tikzpicture}
\end{center}
\caption{Turning $M$ with fixed point $r$ into $T(M)$}
\end{figure}

\begin{lemma}
$\mathcal{X}$ partitions $|M|$.
\end{lemma}
\begin{proof}
By construction
$$
\begin{array}{r c l}
X_i \cap X_j \not= \emptyset & \Rightarrow & i = j \\
Y_k \cap Y_l \not= \emptyset & \Rightarrow & k = l 
\end{array}
$$
so it remains to show that $\mathcal{X}$ covers $|M|$.  We pick an arbitrary $z \in |M|$ and consider $\path{z,r}$, which exists as all CFPOs considered are connected.

Let $z_0 (=z), z_1, \ldots z_n (=r)$ be the endpoints of $\path{z,r}$.  We know that $z_n \in X_0$ as $z_n = r$, and hence $z_{n-1} \not\parallel z_n$ implies that $z_{n-1} \in \bigcup \mathcal{X}$.  Similarly $z_{n-2} \not\parallel z_{n-1}$ implies that $z_{n-1} \in \bigcup \mathcal{X}$ and so on along $\path{z,r}$ until we deduce that $z \in \bigcup \mathcal{X}$.
\end{proof}

If we start with a rooted tree, and use the root for our procedure, our construction returns the original structure with an additional predicate which is realised everywhere.  Our eventual goal is to say that the canonical representative of $M$ is the canonical representative of $T(M)$, and to do so we must show that $T(M)$ is a tree with the same automorphism group as $M$.

This construction has the unfortunate property that we may have to make a choice of fixed point, and the resulting structures depend on this choice.  However, since our claim is that $T(M)$ is a tree, rather than a canonical tree, we may sweep this difficulty under the carpet of Rubin's work.

\begin{prop}
$\langle T(M) , \leq_{T(M)}, U \rangle$ is a tree.
\end{prop}
\begin{proof}
$M$ is connected so $\leq_{T(M)}$ is defined everywhere.

If $s_0, s_1 \leq_{T(M)} t$ then $\lbrace s_0, s_1 \rbrace \subseteq \path{t,r}$, and since $M$ is cycle-free this means that either $s_0 \in \path{s_1,r}$ or $s_1 \in \path{s_0,r}$, showing that $s_0 \not \parallel s_1$, and thus all initial sections of $T(M)$ are linearly ordered.  Finally, $r \in \path{r,t}$ for all $t$, so every pair from $T$ has a common lower bound, showing that $\langle T(M), \leq_{T(M)}, U \rangle$ is a tree.
\end{proof}

Of course, this construction is without merit if it does not preserve the automorphism group.  We work towards that goal with the following lemmas.

\begin{lemma}\label{lemma:MinterpretableinT(M)}
$\langle M, \leq_M, r \rangle$ is interpretable in $\langle T(M), \leq_{T(M)}, U \rangle$.
\end{lemma}
\begin{proof}
The following formulas form an interpretation of $\langle M, \leq_M, r \rangle$ in $\langle T(M), \leq_{T(M)}, U \rangle$:
\begin{enumerate}
\item $\phi_{Dom}(x)$, which defines the domain of the interpretation.  We take
$$x=x$$
\item $\phi_{Eq}(x)$, which defines equivalence classes on the domain of the interpretation.  Again, we take
$$x=x$$
\item A formula $\phi_{\leq_M}(x,y)$.  We take the disjunction of the following clauses:
\begin{enumerate}
\item $(x \leq _T y \wedge \forall z (x \leq_T z \leq_T y \rightarrow U(z) ))$
\item $(y \leq_T x \wedge \forall z (y \leq_T z \leq_T x \rightarrow \neg U(z))) $
\item $(U(y) \wedge \neg U(x)) \wedge$
$$\exists z \left( \begin{array}{l}
 z \leq_{T(M)} \lbrace x,y \rbrace \wedge \\
\forall w (z \leq_{T(M)} w \leq_{T(M)} y \rightarrow U(w)) \wedge \\
\forall w (z \leq_{T(M)} w \leq_{T(M)} x \rightarrow \\
\hspace{20pt}\left( \begin{array}{c}
(U(w) \rightarrow \forall v (z \leq_{T(M)} v \leq w \rightarrow U(v)))\wedge \\
(\neg U(w) \rightarrow \forall v (w \leq_{T(M)} v \leq x \rightarrow \neg U(v)) )
\end{array} \right)
\end{array} 
 \right)$$
\end{enumerate}

\item A formula $\phi_{r}(x)$.  We take
$$
\forall z \neg (z \leq x )
$$
\end{enumerate}
While $\phi_{Dom}$, $\phi_{Eq}$ and $\phi_{r}$ are self-explanatory, to show that $\phi_{\leq_{M}}$ does what is required of it, we examine it clause by clause.

Clause (a) shows that when both $x$ and $y$ lie in the same $X_i$ for some $i$ and $x \leq_{T(M)} y$ then $x \leq_M y$.  Clause (b) shows that when both $x$ and $y$ lie in the same $Y_i$ for some $i$ and $y \leq_{T(M)} x$ then $x \leq_M y$.  Clause (c) covers when $y \in X_i$ and $x \in Y_{i+1} \cup Y_{i-1}$ for some $i$, one instance of which is depicted in Figure 3.  No clause is required for $y \in Y_i$ and $x \not\in Y_i$, because if $x \leq_M y$ then $x \in Y_i$
\end{proof}

\begin{figure}[h!]
\begin{center}
\begin{tikzpicture}[scale=0.13]
\draw (10,-10) -- (10,30);
\draw (20,30) -- (10,10) -- (30,30);
\fill[white] (25,30) ellipse (5 and 2.5);
\draw (25,30) ellipse (5 and 2.5);
\fill (10,20) circle (0.5);
\fill (10,5) circle (0.5);
\fill (17.5,20) circle (0.5);
\draw[anchor=east] (10,20) node {$y$};
\draw[anchor=east] (10,5) node {$z$};
\draw[anchor=south] (17.5,20) node {$x$};
\draw[anchor=south] (10,30) node {$X_i$};
\draw (25,30) node {$Y_{i+1} $};
\draw[anchor=north] (20,-12) node {$T(M)$};

\draw (50,-10) -- (50,30);
\draw (60,-10) -- (50,10) -- (70,-10);
\fill[white] (65,-10) ellipse (5 and 2.5);
\draw (65,-10) ellipse (5 and 2.5);
\fill (50,20) circle (0.5);
\fill (50,5) circle (0.5);
\fill (57.5,0) circle (0.5);
\draw[anchor=east] (50,20) node {$y$};
\draw[anchor=east] (50,5) node {$z$};
\draw[anchor=north] (57.5,0) node {$x$};
\draw[anchor=south] (50,30) node {$X_i$};
\draw (65,-10) node {$Y_{i+1} $};
\draw[anchor=north] (60,-12) node {$M$};
\end{tikzpicture}
\end{center}
\caption{Clause (c) of $\phi_{\leq_{M}}$ in Lemma \ref{lemma:MinterpretableinT(M)}}
\end{figure}
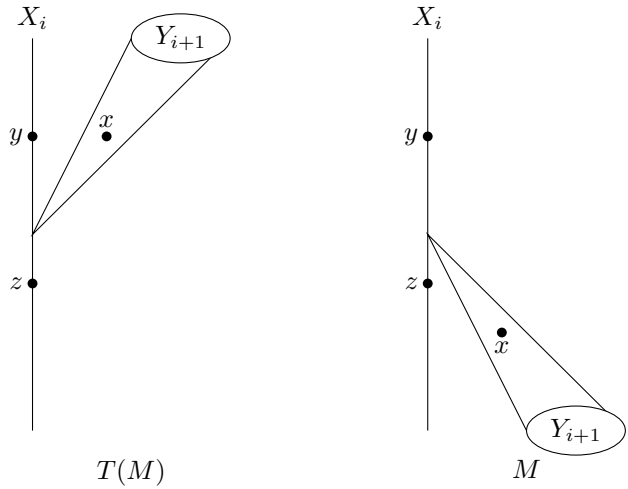\label{FigureClause3}

\begin{lemma}\label{lemma:T(M)andisomorphism}
Suppose $M_0$ and $M_1$ are connected CFPOs with fixed points $r_0$ and $r_1$ respectively.  Then $\langle M_0, \leq_{M_0}, r_0 \rangle \cong \langle M_1, \leq_{M_1}, r_1 \rangle$ if and only if $$\langle T(M_0), \leq_{{T(M_0)}}, U_{T(M_0)} \rangle \cong \langle T(M_1), \leq_{{T(M_1)}}, U_{T(M_1)} \rangle$$
\end{lemma}
\begin{proof}
Since we constructed $\leq_T$ and $U$ using path-betweenness and $\leq_M$, both of which are preserved by isomorphism,
$$
\begin{array}{c}
\langle M_0, \leq_{M_0}, r_0 \rangle \cong \langle M_1, \leq_{M_1}, r_1 \rangle \Rightarrow \\
\langle T(M_0), \leq_{{T(M_0)}}, U_{T(M_0)} \rangle \cong \langle T(M_1), \leq_{{T(M_1)}}, U_{T(M_1)} \rangle
\end{array}
$$
The other direction of the isomorphism is a consequence of the fact that in Lemma \ref{lemma:MinterpretableinT(M)} the domain of the interpretation is $T(M)$ itself.
\end{proof}

This second lemma shows that the construction behaves when we take certain substructures.  We will take from $M$ an extended cone $C$, and show that $T(C)$ is isomorphic to either the corresponding substructure of $T(M)$, or the corresponding substructure with the roles of $U$ and $\neg U$ reversed.

\begin{lemma}\label{lemma:T(M)andsubstructure}
Let $r$ be a fixed point of $M$ and let $x \in M$.  We define
$$N := \lbrace y \in M \: : \: x \in \path{y,r} \rbrace$$
If we add a colour to $N$ which is only realised by $x$ (to ensure that $x$ is a fixed point of $N$ as a structure in its own right), and use $x$ to construct
$$\langle T(N), \leq_{T(N)}, U_{T(N)} \rangle$$ then if $x \in X$ (recall Definition \ref{Def:XandY}) then
$$\langle N , \leq_{T(M)}, U_{T(M)} \rangle \cong \langle T(N), \leq_{T(N)}, U_{T(N)} \rangle$$
otherwise $x \in M \setminus X$ (recall Definition \ref{Def:XandY}) implies that
$$\langle N , \leq_{T(M)}, U_{T(M)} \rangle \cong \langle T(N), \leq_{T(N)}, \neg U_{T(N)} \rangle$$
\end{lemma}
\begin{proof}
This is a simple consequence of the fact that $\path{y,r} = \path{y,x} \cup \path{x,r}$ for all $y \in N$
\end{proof}

\begin{lemma}
The members of $\mathcal{X}$ are preserved setwise by $\aut(M)$.
\end{lemma}
\begin{proof}
All automorphisms fix $r$, so $X_0$, the points greater than $r$, and $Y_0$, the points less than $r$, are fixed setwise.

Let $x_n \in X_n$ and let $y_{n-1} \in Y_{n-1}$ with $y_{n-1} \leq_M x_n $, and assume as an induction hypothesis that for $i < n$ both $X_i$ and $Y_{i}$ are fixed setwise by $\aut(M)$.  Let $\phi \in \aut(M)$ be arbitrarily chosen.  By the induction hypothesis $\phi(y_{n-1}) \in Y_{n-1}$, and since $\phi$ is an automorphism $\phi(y_{n-1}) \leq_M \phi(x_n)$.  If $\phi(x_n) \in \bigcup_{i < n}(X_i \cup Y_i)$ then $\phi^{-1}$ violates the induction hypothesis, so $X_n$ is preserved by $\aut(M)$.  The argument for $Y_n$ is identical.
\end{proof}

\begin{lemma}\label{lemma:TsetwiseX}
$\aut(T)$ preserves the members of $\mathcal{X}$ setwise.
\end{lemma}
\begin{proof}
Let $x \in X_n$.  Since $T \models U(x)$ and $T \models \neg U(y)$ for all $y \in \bigcup Y_i$, we cannot map $x$ to any member of $\bigcup Y_i$.  By taking a witness that $x \in X_n$, and a witness that that witness lies in $Y_{n-1}$ and so on, we obtain a maximal chain $x_1 \leq_{T(M)} x_2 \leq_{T(M)} \ldots x_n (=x)$ such that $U(x_i)$ if and only if $\neg U(x_{i-1})$ and $\neg U(x_{i+1})$, with the additional property that for all $x_i \leq_{T(M)} t \leq_{T(M)} x_{i+1}$ either $[x_i, t]$ or $[t , x_i]$ is monochromatic.

Any automorphism would have to send this chain to a similar chain below the image of $x$, but the length of this chain is determined by $n$, thus all images of $x$ lie in $X_n$.  A similar argument shows the same for $Y_n$, and so we conclude that $\aut(T(M))$ preserves the members of $\mathcal{X}$ setwise.
\end{proof}

\begin{theorem}\label{thm:fixedpoint}
$\aut(\langle M, \leq_M \rangle) \cong_P \aut(\langle T(M), \leq_{T(M)}, U \rangle)$
\end{theorem}
\begin{proof}
Proposition \ref{Lemma:NeededInChapter3} shows that if all the 1- and 2-orbits of $M$ coincide with the 1- and 2-orbits of $T(M)$ then $\aut(T(M)) \cong_P \aut(M)$.  We will start with the 1-orbits, which we will prove by induction on $\mathcal{X}$.

Since $\langle X_0, \leq_M \rangle$ is a tree 
$$\langle X_0, \leq_M \rangle = \langle X_0, \leq_{T(M)} \rangle$$
and since $\langle X_0, \leq_{T(M)}, U_{T(M)} \rangle$ is monochromatic,
$$\aut(\langle X_0, \leq_M \rangle) \cong_P \aut(\langle X_0, \leq_{T(M)}, U_{T(M)} \rangle)$$

From this we conclude that for all $a,b \in X_0$, if $a$ and $b$ lie in different orbits of $M$ but the same orbits of $T$ then
$$\langle \lbrace t \in M \: : \: a \in \path{t,r} \rbrace, \leq_M \rangle \not\cong \langle \lbrace t \in M \: : \: b \in \path{t,r} \rbrace, \leq_M \rangle$$
and
$$
\begin{array}{c}
\langle \lbrace t \in M \: : \: a \leq_{T(M)} t \rbrace, \leq_{T(M)}, U_{T(M)} \rangle \quad \quad  \quad \quad  \quad \\
\cong \\
\quad \quad \quad  \quad \quad  \langle \lbrace t \in M \: : \: b \leq_{T(M)} t \rbrace, \leq_{T(M)}, U_{T(M)} \rangle
\end{array}
$$

However, this contradicts Lemma \ref{lemma:T(M)andsubstructure}, so if $a$ and $b$ lie in the same orbit of $T(M)$ then they lie in the same orbit of $M$.  By symmetry, we also conclude that if $a$ and $b$ lie in the same orbit of $M$ then they lie in the same orbit of $T(M)$.  Similarly, if $a,b \in Y_0$ then $a$ and $b$ lie in the same orbit of $M$ if and only if they lie in the same orbit of $T(M)$.

So now suppose that for $i < n$ the 1-orbits on $X_i$ and $Y_i$ from $\aut(M)$ and $\aut(T(M))$ coincide and let $x,y \in X_n$.  We define, as we did in Lemma \ref{lemma:TsetwiseX}, $x_1, \dots x_n$ and $y_1, \ldots, y_n$, which are linearly ordered by $\leq_{T(M)}$, are the connecting sets of $\path{x,r}$ and $\path{y,r}$ in $\leq_{M}$.

If $x_n$ and $y_n$ belong to the same orbit of $M$ then the automorphism that witnesses this also witnesses that $x_{n-1}$ and $y_{n-1}$ lie in the same orbit of $M$, and hence by our induction hypothesis, the same orbit of $T$.  Since there is an automorphism that maps $x_{n-1}$ to $y_{n-1}$,
$$\langle \lbrace z \in M \: : \: x_{n-1} \in \path{r,z} \rbrace, \leq_M \rangle \cong \langle \lbrace z \in M \: : \: y_{n-1} \in \path{r,z} \rbrace, \leq_M \rangle$$
and hence (using Lemmas \ref{lemma:T(M)andisomorphism} and \ref{lemma:T(M)andsubstructure})
$$
\begin{array}{c}
\langle \lbrace z \in M \: : \: x_{n-1} \in \path{r,z} \rbrace, \leq_{T(M)}, U_{T(M)} \rangle \quad \quad  \quad \quad  \quad \\
\cong \\
\quad \quad \quad  \quad \quad \langle \lbrace z \in M \: : \: y_{n-1} \in \path{r,z} \rbrace, \leq_{T(M)}, U_{T(M)} \rangle
\end{array}
$$
And so there is an isomorphism of $T$ that maps $x_n$ to $y_n$.
The arguments for $x_n, y_n$ being in the same orbit of $T$, and for $x_n, y_n \in Y_n$ are, again, extremely similar, and so omitted.

We now turn out attention to the 2-orbits.  Since $r$ is fixed by both $\aut(M)$ and $\aut(T)$, the 1-orbits can be thought of as 2-orbits where one of the elements is $r$, and the 2-orbits can be thought of as 3-orbits where $r$ is one of the elements.  This viewpoint is exploited to show the coincidence of the 2-orbits of $\aut(M)$ and $\aut(T)$.

Suppose $(x_0,x_1)$ and $(y_0,y_1)$ lie in the same orbit of $M$.  We need only consider the case when $x_0 \in \path{x_1, r}$ as otherwise we can take $x_2$ to be the intersection of $\path{x_0,r}$, $\path{x_0,x_1}$ and $\path{x_1,r}$, and patch automorphisms together around $x_2$.  Note that $x_2$ would be the meet of $x_0$ and $x_1$ in $T(M)$.

There is an automorphism of $M$ that maps $x_0$ to $y_0$, and as we have just seen, this means that
$$\langle \lbrace z \in M \: : \: x_{0} \in \path{r,z} \rbrace, \leq_M \rangle \cong \langle \lbrace z \in M \: : \: y_{0} \in \path{r,z} \rbrace, \leq_M \rangle$$
Since $(x_0,x_1)$ and $(y_0,y_1)$ lie in the same orbit of $M$, there is an isomorphism from
$$\langle \lbrace z \in M \: : \: x_{0} \in \path{r,z} \rbrace, \leq_M \rangle \; \mathrm{to} \; \langle \lbrace z \in M \: : \: y_{0} \in \path{r,z} \rbrace, \leq_M \rangle$$
that maps $x_1$ to $y_1$.  By Lemmas \ref{lemma:T(M)andisomorphism} and \ref{lemma:T(M)andsubstructure} this results in an isomorphism from
$$\langle \lbrace z \in M \: : \: x_{0} \leq_{T(M)} z \rbrace, \leq_{T(M)}, U_{T(M)} \rangle$$
to
$$\langle \lbrace z \in M \: : \: y_{0} \leq_{T(M)} z \rbrace, \leq_{T(M)}, U_{T(M)} \rangle$$
which maps $x_1$ to $y_1$.  We call this isomorphism $\phi$, and we take any automorphism that takes $x_0$ to $y_0$ and call it $\psi$.  The function
$$
\theta(t) := \left\lbrace
\begin{array}{c l}
\phi(t) & t \geq_{T(M)} x_0 \\
\psi(t) & \mathrm{otherwise}
\end{array} \right.
$$
is an automorphism of $T$ which maps $(x_0,x_1)$ to $(y_0,y_1)$, and thus the 2-orbits of $T$ contain the 2-orbits of $M$.

Once again, the argument to show that the 2-orbits of $M$ contain the 2-orbits of $T$ is extremely similar, due to the symmetric nature of Lemmas  \ref{lemma:T(M)andisomorphism} and \ref{lemma:T(M)andsubstructure}, and thus we conclude that the 2-orbits of $M$ and $T$ coincide, and so
$$\aut(\langle M, \leq_M \rangle) \cong_P \aut(\langle T(M), \leq_{T(M)}, U \rangle)$$
\end{proof}

Lots of CFPOs have fixed points, but the CFPOs of the kind discussed in the next lemma reoccur frequently.

\begin{lemma}\label{FixedPointPartition} 
Let $M$ be a connected CFPO.  If there are connected $A,B \subsetneq M$ which are disjoint and fixed setwise by $\aut(M)$ then there are $c,d$ which are fixed points of $M$ and $$\path{A,B}= \path{c,d}$$
\end{lemma}
\begin{proof}
Let $M$ be a connected CFPO, and let $A,B$ be connected proper subsets of $M$ which are disjoint and fixed setwise by $\aut(M)$.  We use the notation $$\path{x,y}^- := \lbrace z \in \path{x,y} \: : \: \exists a,b \in \path{x,y} \; (z=(a\wedge b) \vee z=(a \vee b)) \rbrace$$
In words, $\path{x,y}^-$ are the local maxima and minima of $\path{x,y}$.  Just as with $\path{x,y}$, if $X$ and $Y$ are subsets of $M$ then:
$$
\begin{array}{r c l}
\path{x,Y}^- & := & \bigcap_{y \in Y } \path{x,y}^- \\
\path{X,y}^- & := & \bigcap_{x \in X } \path{x,y}^- \\
\path{X,Y}^- & := & \bigcap_{\substack{ x \in X \\ y \in Y }} \path{x,y}^-
\end{array}
$$
Note that $\path{x,y}^-$ always has finite cardinality.

We are going to find a fixed point using (possibly transfinite) induction.  Fix $b \in B$.

\begin{description}
\item[Base Case] Pick $a_0 \in A$.  We set $c_0 = a_0$ and let $D_0= \lbrace x \in A \: : \: c_0 \in \path{x,b} \rbrace$.
\item[Successor Step] Suppose we have $a_{\alpha-1}$, $c_{\alpha-1}$ and $D_{\alpha-1}$.

Pick $a_\alpha \in A \setminus D_{\alpha-1}$.  Since $b \in \path{c_{\alpha-1}, b}$ and $b \in \path{a_\alpha, b}$,
$$\path{\lbrace c_{\alpha-1}, a_\alpha \rbrace, b} \not= \emptyset$$

Let $$C_\alpha := \lbrace x \in \path{\lbrace c_{\alpha-1}, a_\alpha \rbrace, b} \: : \: |\path{\lbrace c_{\alpha-1}, a_\alpha \rbrace, b}^-| = |\path{x, b}^-| \rbrace$$
$C_\alpha$ is linearly ordered, and is bounded both above and below by elements of $\path{c_{\alpha-1}, b}^- \cup \path{a_\alpha, b}^-$.  Since $M$ is Rubin complete, $C_\alpha$ has both a maximal and a minimum element.

Let $c_\alpha \in C_\alpha$ be such that $\path{\lbrace c_{\alpha-1}, a_\alpha \rbrace, b} = \path{c_\alpha, b}$.

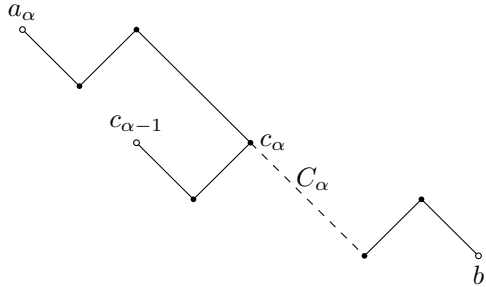
\begin{figure}[ht]
\begin{center}
\begin{tikzpicture}[scale=0.075]
\draw (0,40) -- (10,30) -- (20,40) -- (40,20);
\draw[dashed] (40,20) -- (60,0);
\draw (60,0) -- (70,10) -- (80,0);
\draw (40, 20) -- (30,10) -- (20,20);

\fill[white] (0,40) circle (0.5);
\fill[white] (20,20) circle (0.5);
\fill[white] (80,0) circle (0.5);
\draw (0,40) circle (0.5);
\draw (20,20) circle (0.5);
\draw (80,0) circle (0.5);

\fill (10,30) circle (0.5);
\fill (20,40) circle (0.5);
\fill (30,10) circle (0.5);
\fill (40,20) circle (0.5);
\fill (60,0) circle (0.5);
\fill (70,10) circle (0.5);

\draw[anchor=south] (0,40) node {$a_\alpha$};
\draw[anchor=south] (20,20) node {$c_{\alpha-1}$};
\draw[anchor=west] (40,20) node {$c_\alpha$};
\draw[anchor=north] (80,0) node {$b$};
\draw[anchor=south] (51,10) node {$C_\alpha$};

\end{tikzpicture}
\end{center}
\caption{Finding $c_\alpha$ in Lemma \ref{FixedPointPartition}}
\end{figure}

Since $A$ is connected, $\path{c_{\alpha-1}, a_\alpha} \subseteq A$, and since $c_\alpha \in \path{c_{\alpha-1}, a_\alpha}$, we have that $c_\alpha \in A$.

We define $D_\alpha= \lbrace x \in A \: : \: c_\alpha \in \path{x,b} \rbrace$.  If $D_\alpha = A$ then let $c = c_\alpha$ and stop.

\item[Limit Step] Let $n_\lambda = \min \lbrace |\path{c_\alpha,b}^-| \: : \: \alpha < \lambda \rbrace$.
$$C_\lambda := \lbrace x \in \path{ c_\alpha, b} \: : \: |\path{ c_\alpha, b}^-| = n_\lambda \rbrace$$
$C_\alpha$ is linearly ordered, and is bounded both above and below by elements of $\bigcup_{\alpha < \lambda} \path{c_{\alpha}, b}^-$, so has both a maximal and minimal element.

Let $c_\lambda \in C_\lambda$ be such that $\path{c_\lambda, b} \subseteq \path{\lbrace c_{\alpha}, a_\alpha \rbrace, b}$.  We define $D_\lambda= \lbrace x \in A \: : \: c_\lambda \in \path{x,b} \rbrace$.  If $D_\lambda = A$ then let $c = c_\lambda$ and stop.
\end{description}
We have found a $c$ such that $c \in \path{A,b}$.  If we repeat this induction, fixing $c$ and choosing $b_\alpha$ from $B$ then we find a $d$ such that
$$ \path{c,d} = \path{A,B}$$

Let $\phi \in \aut(M)$.
$$\begin{array}{rcl}
\path{\phi(c), \phi(d)} & = & \phi (\path{c,d}) \\
& = & \phi(\path{A,B}) \\
& = & \path{\phi(A),\phi(B)} \\
& = & \path{A,B} \\
& = & \path{c,d}
\end{array}
$$
Therefore both $c$ and $d$ are fixed by all automorphisms of $M$. 
\end{proof}

\subsection{$\cfpo{n}$}\label{Section:CFPOn}

\begin{lemma}\label{lemma:CFPO3}
If $M$ is a connected $\cfpo{3}$ then $M$ is treelike.
\end{lemma}
\begin{proof}
A $\cfpo{3}$ can be split into three possibly empty sections, a tree which is above a linear order, which in turn is above a reverse ordering of a tree.  If the tree section is empty the reverse tree cannot be empty, and vice versa.

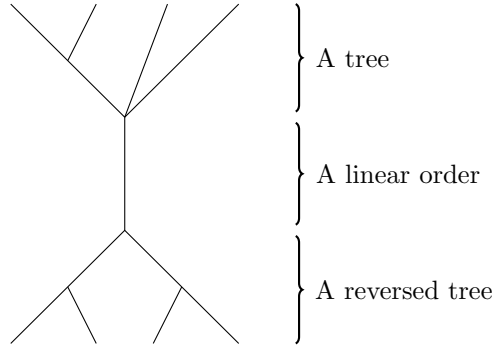
\begin{figure}[h!]
\begin{center}
\begin{tikzpicture}[scale= 0.075]
\draw (0,60) -- (20,40) -- (40,60);
\draw (15,60) -- (10,50);
\draw (20,40) -- (27.5,60);

\draw (20,40) -- (20,20);

\draw (0,0) -- (20,20) -- (40,0);
\draw (15,0) -- (10,10);
\draw (25,0) -- (30,10);

\draw[snake=brace, thick] (50,60)--(50,41);
\draw[snake=brace, thick] (50,39)--(50,21);
\draw[snake=brace, thick] (50,19)--(50,0);

\draw (60,50.5) node {A tree};
\draw (68,30) node {A linear order};
\draw (69,9.5) node {A reversed tree};

\end{tikzpicture}
\end{center}
\caption{A typical $\cfpo{3}$}
\end{figure}

By marking the reversed tree with a unary predicate and reversing its order we obtain a tree which has the same automorphism group as the $\cfpo{3}$.
\begin{figure}[h!]
\begin{center}
\begin{tikzpicture}[scale= 0.075]
\draw (0,40) -- (20,20) -- (40,40);
\draw (15,40) -- (10,30);
\draw (20,20) -- (27.5,40);

\draw (20,20) -- (20,0);

\draw[dashed] (0,20) -- (20,0) -- (40,20);
\draw[dashed] (15,20) -- (10,10);
\draw[dashed] (25,20) -- (30,10);
\end{tikzpicture}
\end{center}
\caption{A Tree with the same Automorphism Group}
\end{figure}
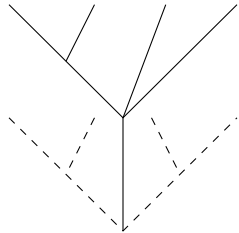
\end{proof}

\begin{theorem}\label{thm:cfpo2n+1}
If $M$ is a connected $\cfpo{2n+1}$ then $M$ is treelike.
\end{theorem}
\begin{proof}
Our strategy is to find a subset of $M$ which is a $\cfpo{3}$ and is fixed setwise by $\aut(M)$, and add cones to the tree corresponding to this $\cfpo{3}$ to obtain a tree with the same automorphism group as $M$.

We consider the $\phi(a_n)$ and $\phi(a_n^*)$, the images in $M$ of the midpoints of $\mathrm{Alt}_{2n+1}$ and $\mathrm{Alt}_{2n+1}^*$ under all possible embeddings $\phi$.  Let $C$ be the set of all such $\phi(a_n)$ and $\phi(a_n*)$.  This is the candidate for the $\cfpo{3}$ we require for our strategy, but first we must show that it is indeed a $\cfpo{3}$, and that it is fixed setwise by $\aut(M)$.

Suppose that $C$ contains an antichain ${x_n,y_n}$.  Since $M$ is connected there must be a path between $x_n$ and $y_n$.  We also pick particular copies of either $\mathrm{Alt}_{2n+1}$ or $\mathrm{Alt}_{2n+1}^*$ that contain $x_n$ and $y_n$, and label the points using $x_i$ and $y_i$ appropriately.  $X$ is the set $\lbrace x_i \rbrace$, while $Y= \lbrace y_i \rbrace$.

To show that the maximum length of a path though $C$ is 3 we consider how the ends of $\path{x_n,y_n}$ interact with $X$ and $Y$.
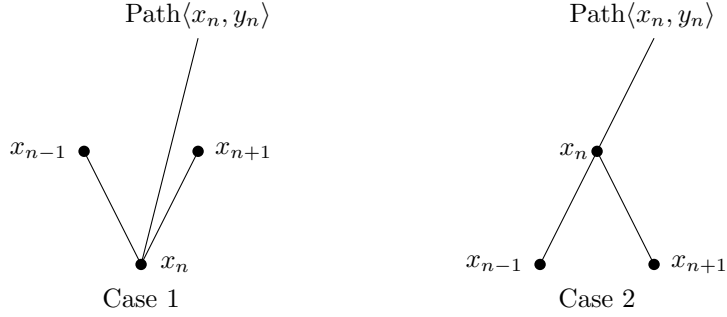
\begin{figure}[h]
\begin{center}
\begin{tikzpicture}[scale=0.15]
\draw (5,10) -- (10,0) -- (15,10);
\draw (10,0) -- (15,20);

\draw (45,0) -- (50,10) -- (55,0);
\draw (50,10) -- (55,20);

\fill (5,10) circle (0.5);
\fill (10,0) circle (0.5);
\fill (15,10) circle (0.5);
\fill (45,0) circle (0.5);
\fill (50,10) circle (0.5);
\fill (55,0) circle (0.5);

\draw (1,10) node {$x_{n-1}$};
\draw (13,0) node {$x_{n}$};
\draw (19,10) node {$x_{n+1}$};
\draw (15,22) node {$\path{x_n,y_n}$};

\draw (41,0) node {$x_{n-1}$};
\draw (48,10) node {$x_{n}$};
\draw (59,0) node {$x_{n+1}$};
\draw (54,22) node {$\path{x_n,y_n}$};

\draw (10,-3) node {Case 1};
\draw (50,-3) node {Case 2};
\end{tikzpicture}
\end{center}
\caption{Interactions between $X$ and $Y$}
\end{figure}

The cases where $x_n$ is an upper point of $\path{x_n,y_n}$ are reverse orderings of Cases 1 and 2, so will not be done explicitly.  Also there is nothing special in our choice of $X$, so these arguments also apply to $Y$.
\paragraph*{Case 1}
In this case $x_n$ is a lower point of both $X$ and $\path{x_n,y_n}$.

If $[x_n,x_{n+1}] \cap \path{x_n,y_n} \not= \emptyset$ then $[x_n,x_{n-1}] \cap \path{x_n,y_n} = \emptyset$, otherwise $x_{n-1}$ and $x_{n+1}$ would be related.  So the union of at least one of $\lbrace x_0, \ldots x_{n-1} \rbrace$ or $\lbrace x_{n+1}, \ldots x_{2n} \rbrace$ with $ \path{x_n,y_n}$ is a copy of a finite section of Alt.

\paragraph*{Case 2}
In this case $x_n$ is an upper point of $X$ but a lower point of $\path{x_n,y_n}$.  As both $x_{n-1}$ and $x_{n+1}$ lie below $x_n$ the two paths $\path{x_{n-1},y_n}$ and \linebreak $\path{x_{n+1},y_n}$ both contain and have the same length as $\path{x_n,y_n}$.  We also know that $x_{n-2}$ cannot be contained in $\path{x_{n-1},y_n}$, as this would require $x_{n-2}$ and $x_n$ to be related.  Similarly $x_{n+2}$ cannot be contained in $\path{x_{n+1},y_n}$.  Thus we see that both $\lbrace x_0, \ldots x_{n-2} \rbrace \cup \path{x_{n-1},y_n}$ and $\lbrace x_{2n}, \ldots x_{n+2} \rbrace \cup \path{x_{n+1},y_n}$ are copies of a finite section of Alt.

Thus in both cases, at least one of $\lbrace x_0, \ldots x_{n-1} \rbrace$ or $\lbrace x_{n+1}, \ldots x_{2n} \rbrace$ with $ \path{x_n,y_n}$ is a copy of a finite section of Alt.  $M$ is a cycle free partial order so, assuming that the configurations of $X$, $Y$ and $\path{x_n,y_n}$ result in the shortest possible finite alternating chain,
$$P := \lbrace x_0, \ldots, x_{n-2} \rbrace \cup \path{x_{n-1},y_{n+1}} \cup \lbrace y_{n+2}, \ldots y_{2n} \rbrace$$
is a copy of a finite section of Alt.  The length of $P$ is $$2n - 2 + \mid \path{x_{n-1},y_{n+1}} \mid$$
By assumption $M$ is a $\cfpo{2n+1}$, so $P$ has at most $2n+1$ elements, thus $\mid \path{x_{n},y_{n}} \mid \; \leq 3$ and $C$ is a $\cfpo{3}$.

To see that $C$ is fixed setwise by automorphisms, simply note for any $x \in C$ and $\phi \in \aut(M)$, the image of the copy of $\alt{2n+1}$ that witnesses the fact that $x \in C$ will witness $\phi(x) \in C$.

We now have the $\cfpo{3}$ our strategy demands, so now we focus on how we may adjoin cones to it to obtain a tree with the same automorphism group as $M$.

For each $x \in C$, we define $B(x):= \lbrace y \in M \: : \: \path{x,y} \cap C = \lbrace x \rbrace \rbrace$.  If we introduce a predicate that fixes $x$ to $B(x)$, then we are able to apply the construction in Definition \ref{Def:XandY} to $B(x)$ using $x$ as the root to obtain $T(B(x))$.  We also know that if there is an automorphism of $M$ that maps $x_0$ to $x_1$ then $B(x_0) \cong B(x_1)$.

For each isomorphism type of $B(x)$, we add a colour predicate $P_x$ to $\langle C, \leq \rangle $ such that $C \models P_x(y)$ if and only if $B(y) \cong B(x)$.  We obtain $\langle C, \leq_M, \lbrace P_x \rbrace \rangle$, a $\cfpo{3}$ such that:
$$\aut(\langle C, \leq_M, \lbrace P_x \rbrace \rangle) \cong_P \lbrace g \in \aut(C) \, : \, \exists h \in \aut(M) \, h |_C = g \rbrace$$
Lemma \ref{lemma:CFPO3} shows that there is a tree, which we call $T(C)$ such that
$$\aut(T(C)) \cong \lbrace g \in \aut(C) \, : \, \exists h \in \aut(M) \, h |_C = g \rbrace$$

We define $T$ to be the structure whose domain is
$$T_C \cup \bigcup_{x \in C} T(B(x))$$
under the equivalence relation that identifies the root of $T_{B(x)}$ with the point of $T_C$ that corresponds with $x$.  We give $T$ the transitive closure of the order inherited from $T_C$ and all the $T_{B(x)}$.  This structure is clearly a tree with the automorphism group of $M$.

Note that this method not only gives a tree $T$ such that $\aut(M) \cong_A \aut(T)$, but also a tree $T$ such that $\aut(T) \cong_P \aut(M)$.
\end{proof}

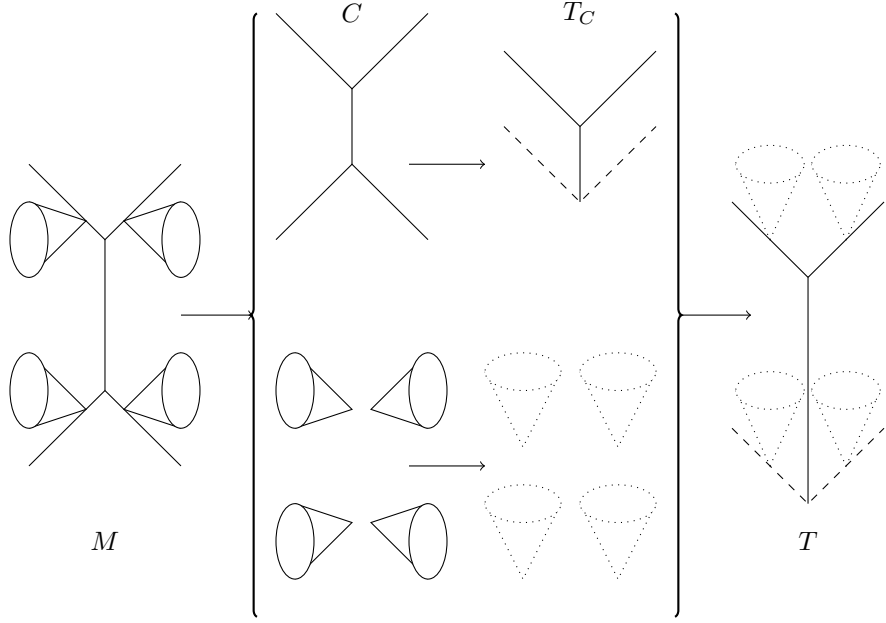
\begin{figure}[ht]
\begin{center}
\begin{tikzpicture}[scale=0.05]
\draw (-10,40) -- (10,20) -- (30,40);
\draw (10,20) -- (10,-20);
\draw (-10,-40) -- (10,-20) -- (30,-40);

\draw (-10,30) -- (5,25) -- (-10,10);
\fill[white] (-10,20) ellipse (5 and 10);
\draw (-10,20) ellipse (5 and 10);

\draw (30,30) -- (15,25) -- (30,10);
\fill[white] (30,20) ellipse (5 and 10);
\draw (30,20) ellipse (5 and 10);

\draw (-10,-30) -- (5,-25) -- (-10,-10);
\fill[white] (-10,-20) ellipse (5 and 10);
\draw (-10,-20) ellipse (5 and 10);

\draw (30,-30) -- (15,-25) -- (30,-10);
\fill[white] (30,-20) ellipse (5 and 10);
\draw (30,-20) ellipse (5 and -10);

\draw[->] (30,0) -- (49,0);
\draw[snake=brace, thick] (50,-80)--(50,80);

\draw (55,80) -- (75,60) -- (95,80);
\draw (75,60) -- (75,40);
\draw (55,20) -- (75,40) -- (95,20);

\draw[->] (90,40) -- (110,40);

\draw (115,70) -- (135,50) -- (155,70);
\draw (135,50) -- (135,30);
\draw[dashed] (115,50) -- (135,30) -- (155,50);

\draw (60,-30) -- (75,-25) -- (60,-10);
\fill[white] (60,-20) ellipse (5 and 10);
\draw (60,-20) ellipse (5 and 10);

\draw (95,-30) -- (80,-25) -- (95,-10);
\fill[white] (95,-20) ellipse (5 and 10);
\draw (95,-20) ellipse (5 and 10);

\draw (60,-50) -- (75,-55) -- (60,-70);
\fill[white] (60,-60) ellipse (5 and 10);
\draw (60,-60) ellipse (5 and 10);

\draw (95,-50) -- (80,-55) -- (95,-70);
\fill[white] (95,-60) ellipse (5 and 10);
\draw (95,-60) ellipse (5 and 10);

\draw[->] (90,-40) -- (110,-40);

\draw[dotted] (110,-15) -- (120,-35) -- (130,-15);
\fill[white] (120,-15) ellipse (10 and 5);
\draw[dotted]  (120,-15) ellipse (10 and 5);

\draw[dotted]  (135,-15) -- (145,-35) -- (155,-15);
\fill[white] (145,-15) ellipse (10 and 5);
\draw[dotted]  (145,-15) ellipse (10 and 5);

\draw[dotted]  (110,-50) -- (120,-70) -- (130,-50);
\fill[white] (120,-50) ellipse (10 and 5);
\draw[dotted]  (120,-50) ellipse (10 and 5);

\draw[dotted]  (135,-50) -- (145,-70) -- (155,-50);
\fill[white] (145,-50) ellipse (10 and 5);
\draw[dotted]  (145,-50) ellipse (10 and 5);

\draw[->] (161,0) -- (180,0);
\draw[snake=brace, thick] (160,80)--(160,-80);

\draw[dashed] (175,-30) -- (195,-50) -- (215,-30);
\draw (195,-50) -- (195,10);
\draw (175,30) -- (195,10) -- (215,30);

\draw[dotted] (176,-20) -- (185,-40) -- (194,-20);
\fill[white] (185,-20) ellipse (9 and 5);
\draw[dotted]  (185,-20) ellipse (9 and 5);

\draw[dotted] (196,-20) -- (205,-40) -- (214,-20);
\fill[white] (205,-20) ellipse (9 and 5);
\draw[dotted]  (205,-20) ellipse (9 and 5);

\draw[dotted] (176,40) -- (185,20) -- (194,40);
\fill[white] (185,40) ellipse (9 and 5);
\draw[dotted]  (185,40) ellipse (9 and 5);

\draw[dotted] (196,40) -- (205,20) -- (214,40);
\fill[white] (205,40) ellipse (9 and 5);
\draw[dotted]  (205,40) ellipse (9 and 5);

\draw (10,-60) node {$M$};
\draw (75,80) node {$C$};
\draw (135,80) node {$T_C$};
\draw (195,-60) node {$T$};
\end{tikzpicture}
\end{center}
\caption{Turning a $\cfpo{2n+1}$ into a Tree}
\end{figure}

\begin{cor}\label{cor:cfpo2n}
If $M$ is a connected $\cfpo{2n}$ then $M$ is treelike.
\end{cor}
\begin{proof}
Let $e \in M$ be an image of $a_0 \in \alt{2n}$ (if $\alt{2n}$ does not embed into $M$ we may consider $M^*$ instead).  Below every point in $\mathrm{Or}(e)$ we adjoin a new point, coloured with a new unary predicate.  This new structure is a $\cfpo{2n+1}$ with the same automorphism group as $M$, so $M$ shares its abstract automorphism group with a tree. 
\end{proof}

While we have found a tree $T$ such that $\aut(M) \cong_A \aut(T)$, and thereby proved the corollary, we may do better than that.  We can delete the points we added to $M$ from $T$ without introducing new automorphisms (as we added these points to every point in an orbit of $M$), getting a $T^*$ such that $\aut(M) \cong_P \aut(T^*)$.

\subsection{$\cfpo{\omega}$}

\begin{theorem}\label{thm:cfpoomega}
If $M$ is a connected $\cfpo{\omega}$ then $M$ is tree-like.
\end{theorem}
\begin{proof}
This proof works in a similar fashion to the proofs of Theorem \ref{thm:fixedpoint}, Lemma \ref{lemma:CFPO3} and Theorem \ref{thm:cfpo2n+1}; by altering the order on the CFPO we produce a tree, while maintaining the automorphism group.  Let $M$ be a Rubin-complete CFPO.

We say that $A \subseteq M$ is a maximal copy of either $\alt{\omega}$ or $\alt{\omega}^*$ if
\begin{itemize}
\item $A$ is the image of $\alt{\omega}$ (or $\alt{\omega}^*$ respectively).
\item There is no image of $\alt{\omega}$ or $\alt{\omega}^*$ that properly contains $A$.
\end{itemize}

Every copy of $\alt{\omega}$ is contained in a maximal copy of either $\alt{\omega}$ or $\alt{\omega}^*$.  To see this, let $\lbrace A_n \subseteq M \: : \: n \in \omega \rbrace$ be such that each $A_n$ is isomorphic to either $\alt{\omega}$ or $\alt{\omega}^*$ and if $n < m$ then $A_n \subsetneq A_m$.  This means that
$$\bigcup_{n \in \mathbb{N}} (A_n \setminus A_0) \cong \alt{\omega} \:\textnormal{or} \: \alt{\omega}^* $$
and therefore
$$A_0 \cup \bigcup_{n \in \mathbb{N}} (A_n \setminus A_0) \cong \alt{}$$

We now describe a procedure that transforms $M$ into a tree while preserving its automorphism group.  Again, we add a unary predicate $U$ to remind us when we've changed direction.
\begin{enumerate}
\item Let $M_0$ be the following set:
$$\lbrace x \in M  :x \; \textnormal{is the first element of a maximal copy of either} \, \alt{\omega} \; \textrm{or} \; \alt{\omega}^* \rbrace$$

If $x \in M_0$ is witnessed by a maximal copy of $\alt{\omega}$ then $x \in M_0$ cannot be witnessed by a maximal copy of $\alt{\omega}^*$.  To see this, let $\lbrace x, a_1, \ldots \rbrace$ be a maximal copy of $\alt{\omega}$ and let $\lbrace x, b_1, \ldots \rbrace$ be a maximal copy of $\alt{\omega}^*$.

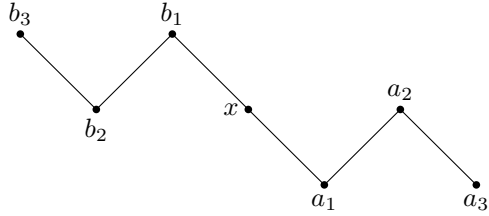
\begin{figure}[ht]
\begin{center}
\begin{tikzpicture}[scale=0.1]
\draw (0,10) -- (10,0) -- (20,10) -- (40,-10) -- (50,0) -- (60,-10) ;
\fill (   0, 10) circle (0.5);
\fill ( 10,   0) circle (0.5);
\fill ( 20, 10) circle (0.5);
\fill ( 30,   0) circle (0.5);
\fill ( 40,-10) circle (0.5);
\fill ( 50,   0) circle (0.5);
\fill ( 60,-10) circle (0.5);

\draw[anchor=south] (   0, 10) node {$b_3$};
\draw[anchor=north] ( 10,   0) node {$b_2$};
\draw[anchor=south] ( 20, 10) node {$b_1$};
\draw[anchor=east] ( 30,   0) node {$x$};
\draw[anchor=north] ( 40,-10) node {$a_1$};
\draw[anchor=south] ( 50,   0) node {$a_2$};
\draw[anchor=north] ( 60,-10) node {$a_3$};
\end{tikzpicture}
\end{center}
\caption{Witnessing $x \in M_0$}
\end{figure}

$b_1 > a_1$, but $b_2 || a_1$, as $b_2 || x$, so $\lbrace b_3, b_2, b_1, a_1, \ldots \rbrace$ is a copy of $\alt{\omega}$, contradicting the assumption that $\lbrace x, a_1, \ldots \rbrace$ was a maximal copy of $\alt{\omega}$.

Let $\sim_C$ be the relation on $M_0$ given by
$$x \sim_C y \Leftrightarrow \left(
\begin{array}{c}
\lbrace x, a_1, \ldots \rbrace \:\textnormal{witnesses}\: x \in M_0  \\
\textnormal{if and only if} \\
\lbrace y, a_1, \ldots \rbrace \:\textnormal{witnesses}\: y \in M_0.
\end{array}\right)$$

That $\sim_C$ is an equivalence relation is readily apparent.  We denote the $\sim_C$-equivalence classes as $C^0_i$.

Let $x \in M_0$, and let this be witnessed by $\lbrace x, a_1, \ldots \rbrace$, a copy of $\alt{\omega}$.  For every $y \in [x]_{\sim_C}$, we know that $y > a_1$, and thus $[x]_{\sim_C} \cup a_1$ is a tree.  Similarly, if $x \in M_0$ is witnessed by a copy of $\alt{\omega}^*$ then $[x]_{\sim_C}$ is a reverse ordered tree.

Let $\lbrace C^0_i \rbrace$ be the set of $\sim_C$-equivalence classes of $M_0$.
\item Assume we have defined $M_{n-1}$ and the $C^{n-1}_i$s. We define $M_n$ to be:
$$\left\lbrace x \in M \setminus \bigcup\limits_{i <n} M_{i} \; : \; 
\begin{array}{c}
x \; \textnormal{is the first element of}\: A \:  \textnormal{which is a maximal}\\
\textnormal{copy of either} \, \alt{\omega} \; \textrm{or} \; \alt{\omega}^*  \textnormal{in} \left( M \setminus \bigcup\limits_{i <n} M_{i} \right)
\end{array}
\right\rbrace$$
Again, $M_n$ is a disjoint union of trees and reverse ordered trees, which we call $C^{n}_i$.
\end{enumerate}
If $C^n_i$ is a tree then $T(C^n_i):= \langle  C^n_i, \leq, U \rangle$ where $U$ is realised nowhere, and if $C^n_i$ is a reverse ordered tree then $T(C^n_i):= \langle (C^n_i)^*, \leq, U \rangle$ where $U$ is realised everywhere.

We define $T_0$ to be the disjoint union of $\lbrace T(C^0_i) \rbrace$ with no new relations added to the ordering.  If we have already defined $T_{n-1}$ then
$$T_n := T_{n-1} \cup \bigcup \lbrace T(C^n_i) \rbrace$$
We add to the order inherited from $T_{n-1}$ and $T(C^n_i)$ pairs of the form $(x,y)$ where
$$x \in T(C^n_i) \: \textnormal{for some} \: i$$
and $y$ is in $T(C^{n-1}_j)$, where $C^{n-1}_j$ is a cone of $x$.  We then take the transitive closure to obtain an ordering.

Put $T(M) := \bigcup\limits_{n \in \mathbb{N}} T_n$.  Since the $M_n$ partition $M$ and since at each stage we place trees above elements of trees, $T(M)$ is a tree.

If $\aut(M)$ does not preserve the $M_n$ then we would have a map that sends a maximal copy of $\alt{\omega}$ or $\alt{\omega}^*$ to a non-maximal copy.  $T(M)$ realises $U$ in monochromatic convex subsets.  In the tree obtained by collapsing each of those subsets to a singleton, every maximal chain is isomorphic to $\omega^*$, so $T(M)$ preserves the $M_n$ set-wise too.

Since each $T(C^n_i)$ is monochromatic, and is order-isomorphic to either $C^n_i$ or $(C^n_i)^*$, if $\aut(T(M)) \not= \aut(M)$ then we must either:

\begin{enumerate}
\item be unable send a $T(C^n_i)$ to a $T(C^n_j)$ where we can map $C^n_i$ to a $C^n_j$; or
\item be able to send $T(C^n_i)$ to $T(C^n_j)$ where we cannot map $C^n_i$ to $C^n_j$.
\end{enumerate}

If $T(C^n_i) \cong T(C^n_j)$ but we cannot map one to the other using an automorphism of $T(M)$ then we must eventually attach $T(C^n_i)$ to something different to what we attach $T(C^n_j)$ to, but then $C^n_i$ emanates from a point that is different to the point that $C^n_i$ emanates from, and we cannot map $C^n_i$ to $C^n_i$.

If we do this argument in reverse we obtain point 2.

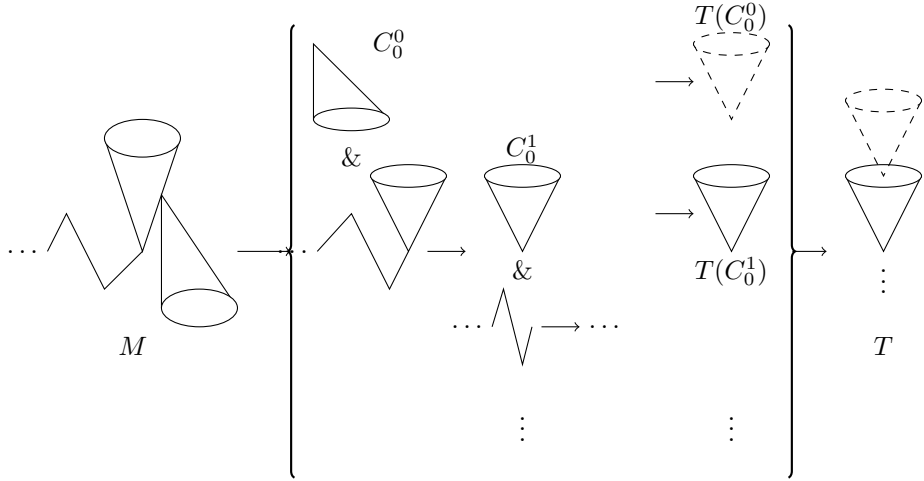
\begin{figure}[ht]
\begin{center}
\begin{tikzpicture}[scale=0.05]
\draw (0,30) -- (10,0) -- (20,30);
\fill[white] (10,30) ellipse (10 and 5);
\draw (10,30) ellipse (10 and 5);

\draw (15,-15) -- (15,15) -- (35,-15);
\fill[white] (25,-15) ellipse (10 and 5);
\draw (25,-15) ellipse (10 and 5);

\draw (10,0) -- (0,-10) -- (-10,10) -- (-15,0);
\draw[anchor=east] (-15,0) node {$\ldots$};

\draw[->] (35,0) -- (49,0);
\draw[snake=brace, thick] (50,-60)--(50,60);

\draw (55,35) -- (55,55) -- (75,35);
\fill[white] (65,35) ellipse (10 and 3);
\draw (65,35) ellipse (10 and 3);

\draw (65,25) node {$\&$};

\draw[anchor=east] (56,0) node {$\ldots$};
\draw (56,0) -- (65,10) -- (75,-10) -- (80,0);
\draw (70,20) -- (80,0) -- (90,20);
\fill[white] (80,20) ellipse (10 and 3);
\draw (80,20) ellipse (10 and 3);

\draw[->] (85,0) -- (95,0);

\draw (100,20) -- (110,0) -- (120,20);
\fill[white] (110,20) ellipse (10 and 3);
\draw (110,20) ellipse (10 and 3);

\draw (110,-5) node {$\&$};

\draw[anchor=east] (102,-20) node {$\ldots$};
\draw (102,-20) -- (105,-10) -- (110,-30) -- (112.5,-20); 
\draw[->] (115,-20) -- (125,-20);
\draw[anchor=west] (125,-20) node {$\ldots$};

\draw (110, -45) node {$\vdots$};
\draw (165, -45) node {$\vdots$};

\draw[->] (145,45) -- (155,45);
\draw[dashed] (155,55) -- (165,35) -- (175,55);
\fill[white] (165,55) ellipse (10 and 3);
\draw[dashed] (165,55) ellipse (10 and 3);

\draw[->] (145,10) -- (155,10);
\draw (155,20) -- (165,0) -- (175,20);
\fill[white] (165,20) ellipse (10 and 3);
\draw (165,20) ellipse (10 and 3);

\draw[->] (181,0) -- (190,0);
\draw[snake=brace, thick] (180,60)--(180,-60);

\draw (195,20) -- (205,0) -- (215,20);
\fill[white] (205,20) ellipse (10 and 3);
\draw (205,20) ellipse (10 and 3);
\draw[anchor=north] (205,2) node {$\vdots$};
\draw[dashed] (195,40) -- (205,20) -- (215,40);
\fill[white] (205,40) ellipse (10 and 3);
\draw[dashed] (205,40) ellipse (10 and 3);

\draw (7.5,-25) node {$M$};
\draw (75,55) node {$C^0_0$};
\draw[anchor=south] (110,21) node {$C^1_0$};
\draw[anchor=south] (165,56) node {$T(C^0_0)$};
\draw[anchor=north] (165,1) node {$T(C^1_0)$};

\draw (205,-25) node {$T$};
\end{tikzpicture}
\end{center}
\caption{Turning a $\cfpo{\omega}$ into a Tree}
\end{figure}

Therefore every Rubin-complete $\cfpo{\omega}$ is treelike.  Let $\langle M, \leq_M \rangle$ be a not necessarily Rubin complete $\cfpo{\omega}$, with Rubin completion $\langle M^R, \leq_M, I \rangle$.  There is a tree $T(M)^R$ such that
$$\aut(\langle M^R, \leq_M, I \rangle) \cong_P \aut(\langle T(M^R), \leq_T, I, U \rangle)$$
We define $T(M):= \lbrace x \in T(M^R) \: : \: T(M^R) \models \neg I(x) \rbrace$.  Then
$$\aut(\langle M, \leq_M \rangle) \cong_P \aut(\langle T(M), \leq_T, U \rangle)$$
\end{proof}

\subsection{Disconnected CFPOs}

While this section has only proved results about connected CFPOs, they are readily extended to disconnected CFPOs.

\begin{prop}\label{prop:FixedPointDisconnected}
Let $M$ be a possibly disconnected CFPO with connected components $A_i$, where the $i$ are indexed by $I$.  If $A_i$ is treelike for all $i \in I$ then $M$ is treelike.
\end{prop}
\begin{proof}
For all $i \in I$, let $\langle T(A_i), \leq, U \rangle$ be the coloured tree such that $\aut(\langle T(A_i), \leq_i, U_i \rangle) \cong_A \aut(A_i)$.

$\mathcal{T} := \langle \lbrace r \rbrace \cup \bigcup(T(A_i)), \leq_T, U_T \rangle$
where

\[
\begin{array}{rcl}
\mathcal{T} \models (x \leq_T y) & \Leftrightarrow & (( \exists i \in I \: (x \leq_i y)) \vee (x=r)) \\
\mathcal{T} \models U_T(x) & \Leftrightarrow & \exists i \in I \: U_i(x) 
\end{array}
\]

$\aut(M) \cong_A \aut(\mathcal{T})$, as each of the cones of $r \in \mathcal{T}$ share an automorphism group with its corresponding $A_i$, and may only be mapped to one another by an automorphism of $\mathcal{T}$ if their corresponding $A_i$ are isomorphic.
\end{proof}

\begin{remark}\label{remark:disconnectedinterpretation}
If each of the $T(A_i)$ are obtained using Definition \ref{Def:XandY}, then we may adapt the interpretation in Lemma \ref{lemma:MinterpretableinT(M)} by changing $\phi_{Dom}$ to $x \not= r$ to obtain an interpretation of $\langle M, \leq_M  \rangle$ in $\mathcal{T}$.
\end{remark}

\section{CFPOs in Model Theory}\label{CFPOsMT}

The theory of trees is known to have certain model theoretic properties.  Parigot showed in 1982 that the theory of trees is NIP, and classified the stable ones \cite{Parigot1982}, while Simon showed in 2011 that the theory of trees is inp-minimal \cite{Simon2011}.  The observations that have been made in this section give an easy method for extending these results to the theory of CFPOs.

\subsection{NIP and Trees}

\begin{dfn}
A formula $\phi(\bar{x},\bar{y})$ is said to have the \textbf{independence property} (for a complete theory $T$) if in every model $M$ of $T$ there is, for each $n<\omega$, a family of tuples $\bar{b}_0, \bar{b}_1, \ldots \bar{b}_{n-1}$ such that for every $I \subseteq \lbrace 0, 1, \ldots n-1 \rbrace$ there is some tuple $\bar{a} \in M$ such that
$$M \models \phi(\bar{a},\bar{b}_i) \Leftrightarrow i \in I$$
$T$ is said to be \textbf{NIP} if no formula in $T$ has the independence property.
\end{dfn}

Note that if $T$ is interpretable in $S$ then if $\phi$ has the independence property for $T$ then the interpretation of $\phi$ has the independence property for $S$.  This means that if $T$ is interpretable in $S$ and $S$ is NIP, then $T$ is NIP.

The `headline' result of \cite{Parigot1982} does not mention NIP.

\begin{theorem}[Parigot, Theorem 2.6 of \cite{Parigot1982}]
A type over a tree never has more than $2^{\aleph_0}$ coheirs.
\end{theorem}

`Coheirs' were defined by Poizat, appearing in \cite{Poizat1981} in 1981, the year before Parigot's paper was published.  If you wish to read the proof of this theorem, but find Poizat's French too daunting, then I recommend the seminar notes of Casanovas \cite{Casanovas}, which are in English.  I am not aware of any publicly available English translation or account of Parigot's paper.

\begin{dfn}[Poizat, \cite{Poizat1981}]
Let $M,N$ be models such that $M \prec N$.  Let $p(x) \subseteq q(x)$ where $q \in S_1(N)$ and $p \in S_1(M)$.  We say that $q$ is a \textit{coheir} of $p$ if $q$ is finitely satisfiable in $M$.
\end{dfn}

\begin{theorem}[Poizat, \cite{Poizat1981}]
Let $T$ be a theory.
\begin{enumerate}
\item If $T$ has the NIP then for all $M$ such that $T \models M$ and $|M| = \lambda \geq |T|$, for all $p \in  S_1(M)$ there are at most $2^\lambda$ coheirs of $p$.
\item If $T$ has the IP then for every $\lambda \geq |T|$ there is an $M$ such that $T \models M$ and $|M| = \lambda \geq |T|$, and there is $p \in  S_1(M)$ such that $p$ has $2^{2^\lambda}$ coheirs.
\end{enumerate}
\end{theorem}

Parigot's results do not stop with trees, however.  He extends to `arborescent' structures, defined by Schmerl.

\begin{dfn}[Schmerl \cite{SchmerlArborescent}]
Let $\mathcal{L} = \langle R_0, \ldots, R_{m-1}, U_0, \ldots ,U_{n-1} \rangle$ be a finite language where each $R_i$ is a binary predicate and each $U_i$ is a unary predicate.

Let $(x,y) \equiv (u,v)$ by the following quaternary formula:
$$x \not= y \wedge u \not= v \wedge \bigwedge_{i<m} \left( (R_i(x,y) \leftrightarrow R_i(u,v)) \wedge (R_i(y,x) \wedge R_i(v,u)) \right)$$

Let $M$ be an $\mathcal{L}$-structure.  $M$ is said to be \textbf{arborescent} if for all finite $B \subseteq M$, if $|B| \geq 2$ then there are distinct $a,b \in B$ such that if $c \in B \setminus \lbrace a,b \rbrace$ then $(a,c) \equiv (b,c)$
\end{dfn}

Finitely coloured trees are examples of arborescent structures.

\begin{prop}[Parigot, Corollary 2.8 of \cite{Parigot1982}]
All arborescent structures are NIP.
\end{prop}

\subsection{inp-minimality and Trees}

\begin{dfn}[Shelah, Definition 7.3 of \cite{ShelahClassification}]
An independence pattern (an inp-pattern) of length $\kappa$ is a sequence of pairs $(\phi^\alpha(x,y), k^\alpha)_{\alpha < \kappa}$ of formulas such that there exists an array $\langle a_i^\alpha \: : \: \alpha < \kappa \: , i < \lambda \rangle$ such that:
\begin{itemize}
\item Rows are $k^\alpha$-inconsistent: for each $\alpha < \kappa$, the set $\lbrace \phi^\alpha(x, a^\alpha_i) : i < \lambda \rbrace$ is $k^\alpha$-inconsistent,
\item Paths are consistent: for all $\eta \in \lambda^\kappa$, the set $\lbrace \phi^\alpha(x,a^\alpha_{\eta(\alpha)}) \: : \: \alpha < \kappa \rbrace$ is consistent.
\end{itemize}
\end{dfn}

Note that if $M$ is interpretable in $N$ then any independence pattern in $M$ is also an independence pattern of $N$.

\begin{dfn}[Goodrick \cite{Goodrick}]
A theory is inp-minimal if there is no inp-pattern of length two in a single free variable.
\end{dfn}

\begin{theorem}[Simon, Proposition 4.7 of \cite{Simon2011}]
If $\langle T, \leq, C_i \rangle$ is a coloured tree then $\mathrm{Th}(\langle T, \leq, C_i \rangle)$ is inp-minimal.
\end{theorem}

\subsection{CFPOs}

How can we apply these results to CFPOs?

Let $M$ be a CFPO with connected components $A_i$, indexed by $I$.  For each $A_i$, pick an $a_i \in A_i$ and introduce a new unary predicate $A$ such that
$$M \models A(x) \Leftrightarrow \exists i \in I \: x = a_i$$
Since we are adding an additional symbol to the language $Th(\langle M, \leq_M \rangle)$ can be interpreted in $Th(\langle M, \leq_M, A \rangle)$ simply by forgetting $A$.

$a_i$ is a fixed point of every $\langle A_i, \leq_M, A \rangle$ so we may invoke Remark \ref{remark:disconnectedinterpretation} to note that $Th(\langle M, \leq_M , A \rangle$ is interpretable in $ Th(\mathcal{T})$.

Therefore every CFPO is interpretable in an NIP, inp-minimal theory, and hence is NIP and inp-minimal.

This shows that if a property that is closed under taking an interpretation is possessed by the theory of coloured trees, then it is possessed by the CFPOs, but the interpretation here is of a special form.  If we are allowed to fix points in a CFPO, we are essentially handling a tree, thus I expect any property of the coloured trees that allows reference to a set of parameters to also be possessed by the CFPOs.

\section{Group Conditions}\label{gpconditions}

\begin{dfn}\label{dfn:dinfinity}
$\dinfinity$, the \textbf{infinite dihedral group}, is the group with the following presentation $\langle \sigma, \tau \, \vert \, \sigma^2 = 1 , \sigma \tau \sigma =  \tau^{-1} \rangle$.
\end{dfn}

How $\dinfinity$ occurs as a subgroup of an automorphism group of a CFPO characterises whether it is treelike or not.  We will first examine how $\dinfinity$ can act on trees.

\subsection{Dendromorphic Groups }

\begin{dfn}
If $T$ is a tree that contains points $a$ and $b$ then $$B(a;b) := \lbrace t \in T \, : \, a< t \wedge b \rbrace$$
$B(a;b)$ is the cone of $a$ that contains $b$.  If $a \not< b $ then $B(a;b) = \emptyset$.  If $B$ is a set such that $a \leq B$ then
$$B(a,B):= \bigcup_{b \in B} B(a;b)$$
\end{dfn}

\begin{dfn}
Given an abstract group $G$ and a permutation group $(H,S,\mu(h,s))$ their wreath product, written as $G \wr_{S} H$, is the abstract group on domain
$$H \times \lbrace \eta : S \rightarrow G \rbrace$$
We use $\eta(s)$ to denote the function $s \mapsto \eta(s)$, and so $\eta(s_0 s)$ is the function $s \mapsto \eta(s_0 s)$.  The group operation of $G \wr_{S} H$ is given by
$$(h_0, \eta_0(x))(h_1, \eta_1(x))=(h_0 h_1, \eta_0(\mu (h_1^{-1},x)) \eta_1(x))$$

When $G=\aut(M)$ and $H=\aut(N)$ their wreath product $G \wr H$ is the automorphism group of the structure obtained by replacing every element of $N$ with a copy of $G$.
\end{dfn}

\begin{remark}
$\mathbb{Z} \wr \mathbb{Z}_2$ is the automorphism group of the structure obtained by replacing the elements of a 2-element antichain by copies of $(\mathbb{Z},\leq)$, while $\mathbb{Z}_2 \wr \mathbb{Z}$ is the automorphism group of the structure obtained by replacing the elements of $(\mathbb{Z}, \leq)$ with 2-element antichains (the lamplighter group).

\begin{figure}[h]
\begin{center}
\begin{tikzpicture}[scale=0.1]
\draw (10,25) -- (10,0);
\draw (20,25) -- (20,0);

\fill (10,20) circle (0.5);
\fill (10,15) circle (0.5);
\fill (10,10) circle (0.5);
\fill (10,5) circle (0.5);

\fill (20,20) circle (0.5);
\fill (20,15) circle (0.5);
\fill (20,10) circle (0.5);
\fill (20,5) circle (0.5);

\fill (60,20) circle (0.5);
\fill (70,20) circle (0.5);
\fill (60,12.5) circle (0.5);
\fill (70,12.5) circle (0.5);
\fill (60,5) circle (0.5);
\fill (70,5) circle (0.5);

\draw[dashed] (65,25) -- (65,0);

\draw (65,-5) node {$\mathbb{Z}_2 \wr \mathbb{Z}$};
\draw (15,-5) node {$\mathbb{Z} \wr \mathbb{Z}_2$};
\end{tikzpicture}
\end{center}
\caption{Mnemonic for the Wreath Product}
\end{figure}
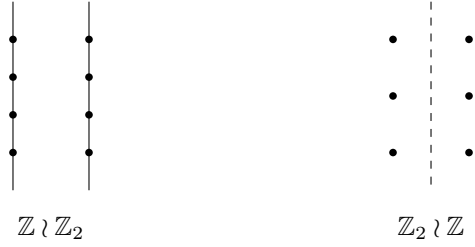

\end{remark}

\begin{remark}
If group $G$ acts on set $X$, with $X_0 \subseteq X$, then $G_{\lbrace X_0 \rbrace}$ is the set-wise stabiliser of $X_0$, while $G_{(X_0)}$ is the point-wise stabiliser.  Similarly for automorphism groups, $\aut_{\lbrace X_0 \rbrace}(M)$ is the set-wise stabiliser of $X_0$ in $M$, while $\aut_{( X_0 )}(M)$ denotes the point-wise stabiliser.  If $X_0 = \lbrace x \rbrace$ then these two notions coincide and we use the pithier expression $G_{x}$ or $\aut_{X_0}(x)$.
\end{remark}

\begin{dfn}\label{dfn:regular}
A tree $T$ is said to be $\textit{regular}$ if:
\begin{enumerate}
\item all the maximal chains are isomorphic to each other;
\item the maximal chains are isomorphic to a finite linear order or $\mathbb{N}$;
\item the ramification order of any non-maximal element of $T$ is at least 2 but finite; and
\item if $|T^{\leq x}|= |T^{\leq y}|$ then the ramification order of $x$ equals the ramification order of $y$.
\end{enumerate}
A tree $T$ is said to be $\textit{fh-regular}$ (finite height) if it is regular and the maximal chains are finite.
\end{dfn}

\begin{figure}
\begin{center}
\begin{tikzpicture}[scale=0.09]

\fill (50,0) circle (0.5);

\fill (25,10) circle (0.5);
\fill (75,10) circle (0.5);
\draw (25,10) -- (50,0) -- (75,10);

\fill (10,20) circle (0.5);
\fill (20,20) circle (0.5);
\draw (20,20) -- (25,10) -- (10,20);
\fill (30,20) circle (0.5);
\fill (40,20) circle (0.5);
\draw (40,20) -- (25,10) -- (30,20);

\fill (60,20) circle (0.5);
\fill (70,20) circle (0.5);
\draw (60,20) -- (75,10) -- (70,20);
\fill (80,20) circle (0.5);
\fill (90,20) circle (0.5);
\draw (90,20) -- (75,10) -- (80,20);

\fill (8,30) circle (0.5);
\fill (12,30) circle (0.5);
\draw (8,30) -- (10,20) -- (12,30);
\fill (18,30) circle (0.5);
\fill (22,30) circle (0.5);
\draw (18,30) -- (20,20) -- (22,30);
\fill (28,30) circle (0.5);
\fill (32,30) circle (0.5);
\draw (28,30) -- (30,20) -- (32,30);
\fill (38,30) circle (0.5);
\fill (42,30) circle (0.5);
\draw (38,30) -- (40,20) -- (42,30);

\fill (58,30) circle (0.5);
\fill (62,30) circle (0.5);
\draw (58,30) -- (60,20) -- (62,30);
\fill (68,30) circle (0.5);
\fill (72,30) circle (0.5);
\draw (68,30) -- (70,20) -- (72,30);
\fill (78,30) circle (0.5);
\fill (82,30) circle (0.5);
\draw (78,30) -- (80,20) -- (82,30);
\fill (88,30) circle (0.5);
\fill (92,30) circle (0.5);
\draw (88,30) -- (90,20) -- (92,30);

\end{tikzpicture}
\end{center}
\caption{Example of a Regular Tree}
\end{figure}
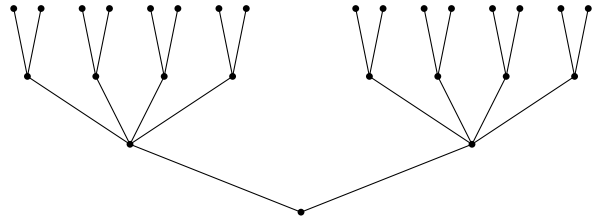

\begin{remark}\label{lemma:fh-regular}
Let $T$ be a finite tree.  $\aut(T)$ acts 1-transitively on the maximal elements of $T$ if and only if $T$ is fh-regular.
\end{remark}

\begin{dfn}
A group $G$ is said to be a \treecauser if it is a Cartesian product of copies at least one of:
\begin{enumerate}
\item $\mathbb{Z} \wr \mathbb{Z}_2$;
\item $\sym$;
\item $\sym \wr \mathbb{Z}_2$; and
\item the automorphism group of a regular tree;
\end{enumerate}
\end{dfn}

Examples of the automorphism group of a regular tree include $S_n$, in particular $\mathbb{Z}_2$,  and $(S_n \wr \mathbb{Z}_2)$.

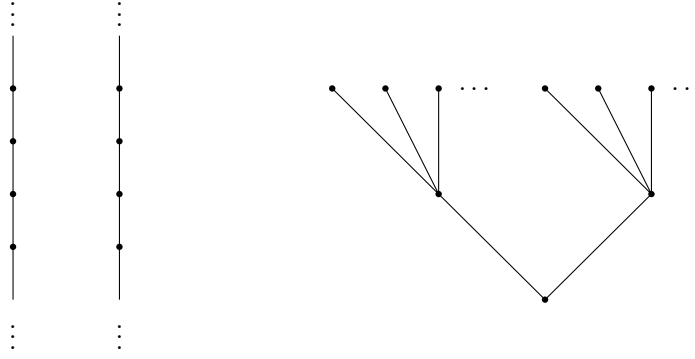
\begin{figure}
\begin{center}
\begin{tikzpicture}[scale=0.14]
\draw (10,25) -- (10,0);
\draw (20,25) -- (20,0);

\fill (10,20) circle (0.3);
\fill (10,15) circle (0.3);
\fill (10,10) circle (0.3);
\fill (10,5) circle (0.3);

\fill (20,20) circle (0.3);
\fill (20,15) circle (0.3);
\fill (20,10) circle (0.3);
\fill (20,5) circle (0.3);

\draw[anchor=north] (10,0) node {$\vdots$};
\draw[anchor=south] (10,25) node {$\vdots$};
\draw[anchor=north] (20,0) node {$\vdots$};
\draw[anchor=south] (20,25) node {$\vdots$};

\draw (70,10) -- (60,0) -- (50,10);
\draw (40,20) -- (50,10) -- (45,20);
\draw (50,20) -- (50,10);

\draw (60,20) -- (70,10) -- (65,20);
\draw (70,20) -- (70,10);

\fill (60,0) circle (0.3);
\fill (70,10) circle (0.3);
\fill (50,10) circle (0.3);
\fill (40,20) circle (0.3);
\fill (45,20) circle (0.3);
\fill (50,20) circle (0.3);
\fill (60,20) circle (0.3);
\fill (65,20) circle (0.3);
\fill (70,20) circle (0.3);

\draw[anchor=west] (51,20) node {$\ldots$};
\draw[anchor=west] (71,20) node {$\ldots$};

\end{tikzpicture}
\end{center}
\caption{Example of Trees whose automorphism group is a \treecauser}
\end{figure}

\begin{dfn}
Let $M$ be a CFPO, let $x \in M$ and let $G \subseteq \aut(M)$.
$$G(x):= \lbrace y \in M \: : \: \exists g \in G \, g(x)=y \rbrace$$
\end{dfn}

\begin{theorem}\label{thm:SuperDinf}
If $T$ is a tree and there exists a $G \leq \aut(T)$ such that $G \cong \dinfinity$ then there exists an $H$ such that $G \leq H \leq \aut(T)$ and $H$ is a \treecauser .
\end{theorem}
\begin{proof}

Let $T$ be a tree such that there is $G \leq \aut(T)$ and $G \cong \dinfinity$.  We use the same presentation of $\dinfinity$ that we gave in Definition \ref{dfn:dinfinity}, so here $\sigma$ and $\tau$ are automorphisms of $T$ that generate $G$ and satisfy the identities $\sigma^2 = 1$ and $\sigma \tau \sigma =  \tau^{-1}$.

Let $t \in T$.  How does $\sigma$ constrain the structure of $G(t)$?  If $t< \sigma (t)$ then $\sigma(t) < \sigma ^2 (t) =t$, which is a contradiction.  Similarly $\sigma (t) < t$ also leads to a contradiction, so if $t \not= \sigma(t)$ then $t \parallel \sigma(t)$.  Since $\sigma \tau \not= \tau \sigma$, we know that $ \supp (\sigma) \cap \supp (\tau)  \not= \emptyset$.

First suppose that $t \in T$ is such that $\lbrace \phi|_{G(t)} \: : \: \phi \in G \rbrace \not\cong \dinfinity$.  This means that there is some $n \in \mathbb{Z}$ and $i \in \lbrace 0,1 \rbrace$ such that $\tau|_{G(t)}^n \sigma|_{G(t)}^i  = id|_{G(t)}$.
\begin{enumerate}
\item If $\sigma|_{G(t)} = id|_{G(t)}$ then the identity $\sigma \tau = \tau^{-1} \sigma$ becomes $\tau = \tau^{-1}$ and we learn that $G(t) = \lbrace t \rbrace$ and $\aut(G(t))$ is trivial.

\item If $\tau|_{G(t)}^n = id|_{G(t)}$ then $G(t)$ is a finite antichain and so $G(t)^+$ is a finite tree whose automorphism group acts transitively on its maximal elements, and by Remark \ref{lemma:fh-regular} is fh-regular, so $\aut(G(t))$ is the automorphism group of the fh-regular tree $G(t)^+$.

\item If $\sigma \tau^n|_{G(t)} = id|_{G(t)}$ then we can deduce that $\sigma|_{G(t)} = \tau^n|_{G(t)}$, and thus $ \tau^{2n}|_{G(t)} = id|_{G(t)}$.
\end{enumerate}

Now we suppose $t \in T$ is such that $\lbrace \phi|_{G(t)} \: : \: \phi \in G \rbrace \cong \dinfinity$.

We now examine the possible action of $\tau$ on $t$.  Since $\tau$ has infinite order, $\lbrace \tau^n (t) \: : \: n \in \mathbb{Z} \rbrace$ and $\lbrace \tau^n \sigma(t) \: : \: n \in \mathbb{Z} \rbrace$ are infinite.  We now consider various cases to deduce the structure of $G(t)$.
\begin{description}

\item[Case 1:]$t < \tau (t)$ or $t > \tau (t)$

Without loss of generality we assume that $t < \tau (t)$.

Since $t< \tau(t)$ we know that $\tau^m(t) < \tau^n(t)$ if and only if $m<n$, where $m,n \in \mathbb{Z}$.  Suppose $\sigma$ fixes one of these $\tau^m(t)$.  Hence
$$\begin{array}{r c l}
\sigma \tau^m (t) & = & \tau^m(t) \\
\end{array}$$
but in $\dinfinity$ we know that $\tau^{-m} \sigma = \sigma \tau^{m}$, so 
$$\begin{array}{r c l}
\tau ^{-m} \sigma(t) & = & \tau^m (t) \\
\sigma(t) & = & \tau^{2m} (t) \\
\end{array} $$
which means that $\sigma$ maps $t$ to $\tau^{2m}(t)$, which in this case is assumed to be greater than $t$, which we have already shown yields a contradiction, and thus $\sigma$ does not fix any $\tau^n(t)$.

We suppose that there is an $n \in \mathbb{Z}$ such that $\tau ^n (t) \leq t \wedge \sigma(t)$.  We know that $\sigma \tau^n (t) \parallel \tau^n(t)$, which is the situation depicted in Figure 14.

\begin{figure}[h]
\begin{center}
\begin{tikzpicture}[scale=0.16]
\draw (0,30)  -- (10,20) -- (20,30);
\draw (10,20) -- (10,10) -- (20,0) -- (30,10);

\fill (0,30) circle (0.3);
\fill (10,20) circle (0.3);
\fill (20,30) circle (0.3);
\fill (10,10) circle (0.3);
\fill (20,0) circle (0.3);
\fill (30,10) circle (0.3);

\draw[anchor=west] (1,30) node {$t$};
\draw[anchor=west] (21,30) node {$\sigma(t)$};
\draw[anchor=west] (11,20) node {$t \wedge \sigma(t)$};
\draw[anchor=west] (11,10) node {$\tau^n(t)$};
\draw[anchor=west] (31,10) node {$\sigma \tau^n(t)$};
\draw[anchor=west] (21,0) node {$\tau^n(t) \wedge \sigma \tau^n(t)$};
\end{tikzpicture}
\end{center}
\caption{Deduced Structure of $G(t)$ if $(\tau ^n (t) \leq t \wedge \sigma(t))$}
\end{figure}
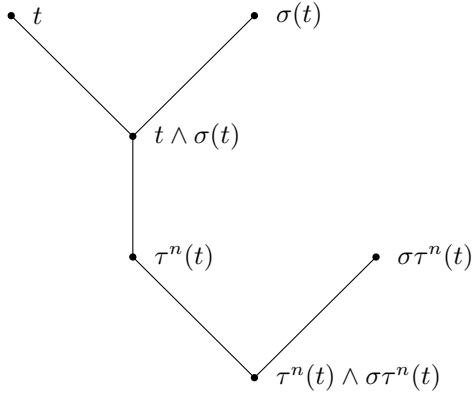

However $\sigma$ maps the pair $(t, \tau^n(t))$ to $(\sigma(t), \sigma \tau^n (t))$, so $\tau^n(t) < t$ implies that $\sigma\tau^n(t) < \sigma(t)$, providing a contradiction.

So there is no $n$ such that $\tau^n (t) \leq t \wedge \sigma(t)$ and then we are in the situation depicted in Figure 15.

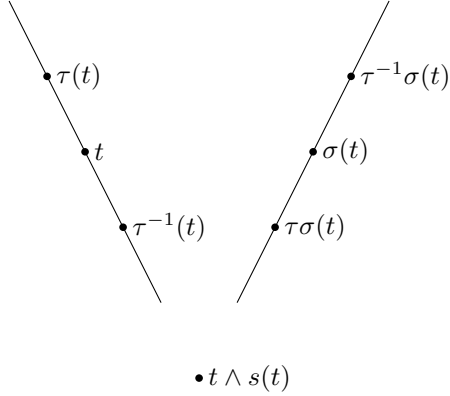
\begin{figure}[h]
\begin{center}
\begin{tikzpicture}[scale=0.1]
\draw (0,50)  -- (20,10);
\draw (30,10) -- (50,50);

\fill (25,0) circle (0.5);

\fill (5,40) circle (0.5);
\fill (10,30) circle (0.5);
\fill (15,20) circle (0.5);

\fill (35,20) circle (0.5);
\fill (40,30) circle (0.5);
\fill (45,40) circle (0.5);

\draw[anchor=west] (25,0) node {$t \wedge s(t)$};

\draw[anchor=west]  (5,40) node {$\tau(t)$};
\draw[anchor=west]  (10,30) node {$t$};
\draw[anchor=west] (15,20) node {$\tau^{-1}(t)$};

\draw[anchor=west]  (45,40) node {$\tau^{-1} \sigma (t)$};
\draw[anchor=west]  (40,30) node {$\sigma (t)$};
\draw[anchor=west]  (35,20) node {$\tau \sigma (t)$};
\end{tikzpicture}
\end{center}
\caption{Deduced Structure of $G(t)$ if $(t \wedge \sigma(t) \leq \tau^{i}(t))$}
\end{figure}

The automorphism group of this structure is clearly $\mathbb{Z} \wr \mathbb{Z}_2$, and so
$$\aut(G(t)) \cong \mathbb{Z} \wr \mathbb{Z}_2$$

\item[Case 2:] $t \parallel \tau(t)$ and $\tau^m(t) \wedge \tau^n(t) = \tau^{m'}(t) \wedge \tau^{n'}(t)$ for all $m \not= n,m' \not= n'$.  We call denote common ramification point,  $\tau^m(t) \wedge \tau^n(t)$ for $m \not= n$, by $x$.  In other words, the $\tau^n(t)$ form an antichain, which ramifies from $x$.

If $x = \sigma(x)$ then the whole orbit of $t$ is an infinite (as $G(t)$ is infinite) antichain  above $x$, and thus $\aut(T)$ is $\sym$.

If $x \not= \sigma(x)$ then the whole orbit of $t$ is two infinite (as both $\lbrace \tau^n (t) \: : \: n \in \mathbb{Z} \rbrace$ and $\lbrace \tau^n \sigma (t) \: : \: n \in \mathbb{Z} \rbrace$ are infinite) antichains , one ramifying from $x$, the other from $\sigma(x)$.  In this case $\aut(T) \cong \sym \wr \mathbb{Z}_2$.

\item[Case 3:] $t \parallel \tau(t)$ and $\tau^m(t) \wedge \tau^n(t) \not= \tau^{m'}(t) \wedge \tau^{n'}(t)$ for some $m,n,m',n'$.

For $m \in \mathbb{N} \setminus \lbrace 0 \rbrace$ let $G_m :=  \lbrace \sigma^{i} \tau^{mn} \: : \: i \in \lbrace 0 , 1 \rbrace \: n \in \mathbb{Z}\rbrace$.  Note that $G_m \cong \dinfinity$.

For brevity's sake, $x_n$ will denote $\tau^{mn}(t) \wedge \tau^{m(n+1)}(t)$.  Suppose that $x_i \not= x_{i+1}$ for all $i$.  Note that $\tau^{mk} (x_n) = x_{n+k}$ because greatest lower bounds are preserved by automorphisms.  For any $i \in \mathbb{Z}$ both $x_i$ and $x_{i+1}$ are below $\tau^{m(i+1)}(t)$, so $\lbrace x_i \: :\: i \in \mathbb{Z} \rbrace$ is linearly ordered and acted on by $\tau^m$, showing that $\tau^m(x_i) < x_i$ or $\tau^m(x_i) > x_i$.

If $\lbrace \phi|_{G_m(x_0)} \: : \: \phi \in G_m \rbrace \not\cong \dinfinity$, then $G_m(x_0)$ is an antichain, but we have just established that $\tau(x_i) < x_i$ or $\tau(x_i) > x_i$, so $\lbrace \phi|_{G_m(x_0)} \: : \: \phi \in G_m \rbrace \cong \dinfinity$, and we may now apply Case 1 to $G_m(x_0)$ and find that $\aut(G_m(x_0)) \cong (\mathbb{Z} \wr \mathbb{Z}_2)$.

Since each $x_i \not= x_{i+1}$, we can deduce the structure depicted in Figure 16.
\begin{figure}[h]\label{figure:cheapcopout}
\begin{center}
\begin{tikzpicture}[scale=0.1]
\draw (0,50)  -- (20,10);
\draw (30,10) -- (50,50);

\fill (25,0) circle (0.5);

\fill (5,40) circle (0.5);
\fill (10,30) circle (0.5);
\fill (15,20) circle (0.5);

\fill (10,50) circle (0.5);
\draw (5,40) -- (10,50);
\fill (15,40) circle (0.5);
\draw (10,30) -- (15,40);
\fill (20,30) circle (0.5);
\draw (15,20) -- (20,30);

\fill (35,20) circle (0.5);
\fill (40,30) circle (0.5);
\fill (45,40) circle (0.5);

\fill (30,30) circle (0.5);
\draw (30,30) -- (35,20);
\fill (35,40) circle (0.5);
\draw (35,40) -- (40,30);
\fill (40,50) circle (0.5);
\draw (40,50) -- (45,40);

\draw[anchor=west] (25,0) node {$x_0 \wedge \sigma (x_0)$};

\draw[anchor=west]  (5,40) node {$x_1$};
\draw[anchor=west]  (10,30) node {$x_0$};
\draw[anchor=west] (15,20) node {$x_{-1}$};

\draw[anchor=west]  (45,40) node {$\sigma (x_1)$};
\draw[anchor=west]  (40,30) node {$\sigma (x_0)$};
\draw[anchor=west]  (35,20) node {$\sigma (x_{-1})$};

\draw[anchor=south]  (10,50) node {$\tau^m(t)$};
\draw[anchor=south]  (15,40) node {$t$};

\draw[anchor=south]  (35,40) node {$\sigma (t)$};
\draw[anchor=south]  (40,50) node {$\tau^m \sigma (t)$};

\end{tikzpicture}
\end{center}
\caption{Deduced Structure needed for Case 3}
\end{figure}
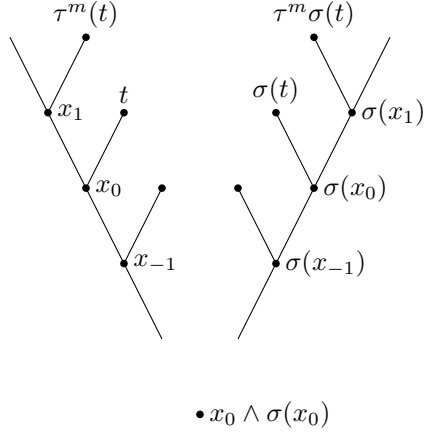

Thus we see that $\aut (G_m(t)) \cong (\mathbb{Z} \wr \mathbb{Z}_2)$.  If we redefine $x_n := \tau^{mn+k}(t) \wedge \tau^{m(n+1) +k}(t)$ and repeat this argument, we see that $\aut (G_m(\tau^k(t))) \cong (\mathbb{Z} \wr \mathbb{Z}_2)$

Let $m_0$ be the least element of the set 
$$\lbrace i = \mathrm{lcm}(n-m,n'-m') \: : \: \tau^m(t) \wedge \tau^n(t) \not= \tau^{m'}(t) \wedge \tau^{n'}(t) \rbrace$$
Note that $\tau^{m_0 n}(t) \wedge \tau^{m_0 (n+1)}(t) \not= \tau^{m_0 (n+1)}(t) \wedge \tau^{m_0 (n+2)}(t)$ for all $n$, so $m_0$ is in fact the least number such that $\aut (G_{m_0}(t)) \cong (\mathbb{Z} \wr \mathbb{Z}_2)$.

$G(t)$ consists of $m_0-1$ copies of $G_{m_0}(t)$, which are preserved by $\sigma$, and $\tau$ acts cyclically on them, and indeed their least elements, which we call $L$.  This gives us $\lbrace \phi|_{L} \: : \: \phi \in G \rbrace \not\cong \dinfinity$, and $\sigma|_L = \mathrm{id}|_L$, so $L$ is trivial and $\aut(G(t)) \cong (\mathbb{Z} \wr \mathbb{Z}_2)$
\end{description}

Therefore for all $t \in T$ the group $\aut(G(t))$ is either trivial or :
\begin{enumerate}
\item $\mathbb{Z} \wr \mathbb{Z}_2$ (from Cases 1 and 3);
\item $\sym$ (from Case 2);
\item $\sym \wr \mathbb{Z}_2$ (from Case 2); or
\item the automorphism group of an fh-regular tree;
\end{enumerate}
each of which is a \treecauser .

We pick one $t \in T$ such that $G(t) \not= \lbrace t \rbrace$, and let $s := \mathrm{inf} (G(t)^+)$.  The next phase of this proof is to show that the additional automorphisms of $\aut(G(t))$ extend to $B(s;G(t))$.  We do this by addressing each of the possibilities in the above list individually.

Let $\lambda \in \aut(G(t)) \setminus G$.  We wish to extend $\lambda$ to $B(s;G(t))$ and show that the group of the extensions of elements of $\aut(G(t))$ is a \treecauser .

\begin{enumerate}
 \item Suppose $\aut(G(t)) \cong ( \mathbb{Z} \wr  \mathbb{Z}_2)$.  Then $\lambda$ is characterised by where it maps $t$ and $\sigma(t)$.  Let's suppose that $\lambda(t) = \tau^n(t)$ and $\lambda(\sigma(t)) = \tau^m \sigma(t)$.  Then we define $\bar{\lambda}$ to be the following:
$$
\bar{\lambda} :x \mapsto \left\lbrace
\begin{array}{l c}
\tau^n(x) & x \in B(s;t) \\
\tau^m(x) & x \in B(s;\sigma(t))
\end{array}
\right.
$$
If $\lambda(t) = \tau^n \sigma(t)$ and $\lambda(\sigma(t)) = \tau^m (t)$ then
$$
\bar{\lambda} :x \mapsto \left\lbrace
\begin{array}{l c}
\tau^m \sigma(x) & x \in B(s;t) \\
\tau^n \sigma(x) & x \in B(s;\sigma(t))
\end{array}
\right.
$$
Thus we may extend $\lambda$ to a unique element of $\aut((B(s;G(t)))$, so
$$\aut((B(s;G(t))) \cong ( \mathbb{Z} \wr  \mathbb{Z}_2)$$

\item Suppose $\aut(G(t)) \cong \sym$.  If there is some $b \in G(t)$ such that $\sigma(b) = b$ and $\sigma|_{B(s;b)} \not= \mathrm{id}|_{B(s;b)}$ then there are two possible extensions of $\lambda$.  If $x \in B(s,a)$ and $\tau^n(a) = \lambda(a) = \tau^m \sigma (a)$ then
$$
\begin{array}{l}
\bar{\lambda}_0 :x \mapsto \tau^n (x) \\
\bar{\lambda}_1 :x \mapsto \tau^m \sigma (x)
\end{array}
$$
Since each $\lambda$ may be extended to two elements of $\aut((B(s;G(t)))$, we know that
$$\aut((B(s;G(t))) \cong ( \mathbb{Z}_2 \times  \sym)$$
Otherwise if $x \in B(s; a)$ and $\lambda(a) = \tau^n \sigma^i (t)$ then
 $$
\bar{\lambda} :x \mapsto \tau^n \sigma^i(x) \\
$$
and we uniquely extend $\lambda$, showing
$$\aut((B(s;G(t))) \cong \sym$$

 \item Suppose $\aut(G(t)) \cong (\sym \wr \mathbb{Z}_2)$.  If $x \in B(s; a)$ and $\lambda(a) = \tau^n \sigma^i (t)$ then
 $$
\bar{\lambda} :x \mapsto \tau^n \sigma^i(x) \\
$$
so we can uniquely extend $\lambda$, showing
$$\aut((B(s;G(t))) \cong (\sym \wr \mathbb{Z}_2)$$

 \item Suppose $G(t)^+$ is an fh-regular tree, and suppose that there is an $x \in B(s;G(t))$ such that $\lbrace \phi|_{G(x)} \: : \: \phi \in G \rbrace \cong \dinfinity$.  Clearly $G$ preserves $B(s;G(t))$, so
 $$G(x) \subseteq  B(s;G(t))$$
 
 Suppose that $x \in B(s;t)$.  Then $\tau^n \sigma^i (x) \in B(s;\tau^n \sigma^i (t))$ for all $n \in \mathbb{Z}$ and $i \in \lbrace 0,1 \rbrace$, therefore for all $y \in G(t)$
 $$G(x) \cap B(s;y) \not= \emptyset$$
  
 Rather than look at $\lambda \in \aut(G(t))$, we instead extend every $\mu \in \aut(G(x))$ to obtain a dendromorphic supergroup of $G$ in $B(s,G(t))$.
 
 Now we suppose that there is no $x \in B(s;G(t))$ such that $\lbrace \phi|_{G(x)} \: : \: \phi \in G \rbrace \cong \dinfinity$.  We will define by induction a family of sets that we will call $X_k$ which will help us extend $\lambda$.
 
 Let $X_0$ be the maximal subset of $B(s,G(t))$ such that for all $\phi, \psi \in G$
 $$\phi|_{G(t)} = \psi|_{G(t)} \Rightarrow \phi|_{X_0} = \psi|_{X_0}  $$
 Let $x \in B(s;y)$ and let $\phi \in G$ be such that $\lambda(y)=\phi(y)$.
   $$
\bar{\lambda} :x \mapsto \phi(x) \\
$$
Since all the possible $\phi$ agree, this map is a well-defined, unique extension of $\lambda$, so $\aut(X_0) \cong \aut(G(t)^+)$.  If $X_0 = B(s;G(t))$ then we have extended $\lambda$ to $B(s;G(t))$ and we are done.

Suppose that we have defined $X_{k-1}$, but $X_{k-1} \not= B(s;G(t))$.  Let $x_k \in B(s,G(t)) \setminus X_{k-1}$.  Let $X_k$ be the maximal subset of $B(s,G(x_1))$ such that for all $\phi, \psi \in G$
 $$\phi|_{G(t)} = \psi|_{G(t)} \Rightarrow \phi|_{X_1} = \psi|_{X_1} $$
 Again, $\aut(X_k) \cong \aut(G(x_k)^+)$ and if $X_k = B(s;G(t))$ then we have extended $\lambda \in \aut(X_k)$ to $B(s;G(t))$ and we are done.
 
 If $X_k \not= B(s;G(t))$ then we define $X:= \bigcup\limits_{k \in \mathbb{N}} X_k$.   We know how to extend $\lambda$ to $X$, so if we can show that:
 \begin{enumerate}
 \item $X = B(s;G(t))$; and
 \item there is a regular tree $F$ such that $\aut(X) = \aut(F)$;
 \end{enumerate}
 then we will have shown that $\aut(B(s,G(t)) \cong \aut(F)$.
 
 \begin{enumerate}
 \item For all $k$, the orbit $|G(x_{k})| > |G(x_{k-1})|$, as there are $\phi, \psi \in G$ such that $\phi(x_{k-1}) = \psi(x_{k-1})$ but $\phi(x_{k}) \not= \psi(x_{k})$, so the set $\lbrace |G(x_{k})| \: : \: k \in \mathbb{N} \rbrace$ is unbounded.
 
  If $y \in B(s;G(t)) \setminus X$ then for all $k$
  $$\tau^{|G(x_{k})|}(y) \not= y$$
  so $G$ acts as $\dinfinity$ on $G(y)$, and we have already seen how to extend $\lambda$ to $B(s;G(t))$ in this case, so we may assume now that $X = B(s;G(t))$.
  \item Since $X_k$ extends $X_{k-1}$ and since $s$ is the root of both $G(x_{k-1})^+$ and $G(x_{k})^+$, we know that $G(x_k)^+$ is an extension of $G(x_{k-1})^+$.  Therefore we consider the tree $F:=\bigcup\limits_{k \in \mathbb{N}} G(x_k)^+$.
  
  Let $(s, y_1 \ldots )$ and $(s, z_1, \ldots )$ denote maximal chains of $F$.  Since each $G(x_k)^+$ is an fh-regular tree, given any two maximal chains of $F$ there is a partial automorphism from the initial $k$ elements of the first to the initial $k$ elements of the second.  The union of all these partial automorphisms will be an automorphism of $F$, and thus $\aut(F)$ acts transitively on every maximal chain, which is Condition 1 of Definition \ref{dfn:regular}.
  
  The initial section of every maximal chain of $F$ finite, so every maximal chain is isomorphic to $\mathbb{N}$, Condition 2 of Definition \ref{dfn:regular}.
  
  If $y \in F$ then $y \in G(x_k)^+$ for some $k$, so the ramification order of any non-maximal element of $F$ is at least 2 but finite, showing that $F$ satisfies Condition 3 of Definition \ref{dfn:regular}.
  
  Finally, if $|F^{\leq y}|= |F^{\leq z}|$ then there is a $k$ such that $y, z \in G(x_k)^+$ and $|(G(x_k)^+)^{\leq y}|= |(G(x_k)^+)^{\leq z}|$, so the fact that $G(x_k)^+$ is fh-regular implies that $F$ satisfies Condition 3 of Definition \ref{dfn:regular}, and is regular.
 \end{enumerate}
 
 Therefore there is a regular tree $F$ such that $\aut(B(s,G(t)) \cong \aut(F)$.
\end{enumerate}

For any $t \in T$ let $s_t$ be the root of $G(t)^+$.  Consider the set
$$ \mathcal{B}:= \lbrace B(s_t;G(t)) \: : \: |G(t)| \not= 1 \rbrace \cup \lbrace \lbrace t \rbrace \: : \: |G(t)| \not= 1 \rbrace $$
Let $H$ be the group of all automorphisms of $T$ that fix every $B \in \mathcal{B}$ setwise.
$$ H = \prod_{B \in \mathcal{B}} \aut(B) $$
Since the Cartesian product of dendromorphic groups is dendromorphic, $H$ is also dendromorphic.  We have already seen that $G$ fixes every $B \in \mathcal{B}$ setwise, so $G \leq H$.
\end{proof}

If you are familiar with automorphism groups as topological groups, you may have realised that in the proof of Theorem \ref{thm:SuperDinf} we are essentially calculating the closure of the copy of $\dinfinity$.  In Theorem \ref{thm:dinfinityinCFPOs} we will see that a CFPO is not treelike if and only if its automorphism group contains a closed copy of $\dinfinity$.

While describing this situation using the language of topological groups might have been more elegant, I prefer this approach as it makes it clear that these properties are recognisable from the abstract group.

\subsection{$\dinfinity$ in CFPOs}

\begin{cor} \label{prop:ALTnotTree}
$\aut(\alt{}) \not\cong \aut(T)$ for all trees $T$.
\end{cor}
\begin{proof}
$\aut(\alt{}) \cong \dinfinity$, so if $\aut(T) \cong \aut(\alt{})$ then the whole automorphism group is a copy of $\dinfinity$, and so cannot be contained in a \treecauser .
\end{proof}

So we've established that $\dinfinity$ can occur as a subgroup of the automorphism group of a CFPO in a different way than it can as a subgroup of the automorphism group of a tree.  The rest of this subsection is devoted to finding out how copies of $\dinfinity$ that aren't contained in a \treecauser can act on a CFPO.

\begin{dfn}\label{dfn:connectionclosure}
Let $M$ be a CFPO.  If $X \subseteq M$ then $X^{cc}$, the \textbf{connection closure} of $X$, is the following set
$$\bigcup_{x,y \in X} \path{x,y}$$
In particular, if $G \leq \aut(M)$ and $x \in M$ then this combines with the notation of Definition \ref{dfn:G(t)} to give:
$$G(x)^{cc}:= \bigcup\limits_{g,h \in G} \path{g(x),h(x)}$$
\end{dfn}

\begin{theorem}\label{thm:dinfinityinCFPOs}
Let $M$ be a Rubin complete CFPO and let $G \leq \aut(M)$.  If $G \cong \dinfinity$ then either $G$ is contained in a \treecauser or $G$ acts on a copy of Alt in $M$, but not both.
\end{theorem}
\begin{proof}
If $M$ is a $\cfpo{n}$ for some $n \in \mathbb{N}$ or a $\cfpo{\omega}$ then by Theorem \ref{thm:cfpo2n+1}, Corollary \ref{cor:cfpo2n} and Theorem \ref{cor:cfpo2n} there is a tree $T$ such that $\aut(M) \cong \aut(T)$.  Thus Theorem \ref{thm:SuperDinf} shows that $G$ is contained in a \treecauser and $G$ cannot act on a copy of Alt, as $M$ does not contain a copy of Alt.  We now suppose that $M$ is a connected $\cfpo{\infty}$.

If $G$ fixes $a \in M$ then $G \leq \aut_{a}(M)$.  By adding a colour predicate to $M$ that only $a$ realises, we find a CFPO with a fixed point whose automorphism group is $\aut_{a}(M)$.  Since this CFPO has a fixed point it is treelike (Theorem \ref{thm:fixedpoint}), and Theorem \ref{thm:SuperDinf} shows that there is a \treecauser $X$ which is contained in $\aut_{a}(M)$ and contains $G$.  Therefore if $M \setminus \supp(G) \not= \emptyset$ then $G$ is contained in a \treecauser{.}

Now suppose that $G$ has no fixed point and that $G(m)^{cc}$ is not a $\cfpo{\infty}$ for any $m \in M$.

We can view the connected components of $M \setminus G(m)^{cc}$ as extended cones of elements of $G(m)^{cc}$.  For all $a \in G(m)^{cc}$
$$C(a) := \lbrace x \in M \setminus G(m)^{cc} \: : \: a \in \path{x, G(m)^{cc}} \rbrace$$
i.e. $C(a)$ is the union of all the extended cones of $M \setminus G(m)^{cc}$ that ramify from $a$.  If $\phi \in \aut(\langle G(m)^{cc}, \leq_M \rangle)$ does not extend to an automorphism of $M$ then $\phi$ must map $a$ to $b$ but $C(a) \not\cong C(b)$.

If for all CFPOs $C$ such that $\exists a \in G(m)^{cc} \: C \cong C(a)$ we introduce a colour predicate $P_C$ to $\langle G(m)^{cc}, \leq_M \rangle$ such that
$$\langle G(m)^{cc}, \leq_M \rangle \models P_C(a) \: \Leftrightarrow \: C(a) \cong C$$
Every automorphism of $\langle G(m)^{cc}, \leq_m, P_C \rangle$ is a restriction of an automorphism of $M$.

Each $G(m)^{cc}$ is $G$-invariant, as otherwise we would be able to map a path inside $G(m)^{cc}$ to one outside by an element of $G$, but this map must take the endpoints of this path with it, and these endpoints are elements of $G(m)^{cc}$.

We choose one $m \in M$.  Since $G(m)^{cc}$ is not a $\cfpo{\infty}$, it is treelike.  All of the extended cones that are contained in $M \setminus G(m)^{cc}$ are treelike if we fix the point in $G(m)^{cc}$ that they emanate from, so by replacing $G(m)^{cc}$ and the extended cones, we may find a tree $T$ such that $G \leq \aut(T) \leq \aut(M)$, and so $G$ is contained in a \treecauser{.}

So now suppose that $a \in M$ is such that $G(a)^{cc}$ is a $\cfpo{\infty}$.  From such an $a$ we define
$$b := \path{a,\path{\tau^{-1}(a),\tau(a)}}$$
(because $\path{\tau^{-1}(b),b} \cap \path{b,\tau(b)} = \lbrace b \rbrace$) and consider $G(b)^-$, the set of maximal and minimal points of $G(b)^{cc}$.  If $\tau^{\mathbb{Z}}(b) = G(b)$ then $G(b)^-$ is a copy of Alt on which $G$ acts.

If $\sigma(b) \not\in \tau^{\mathbb{Z}}(b)$ then we consider $G$'s action on
$$\path{\tau^{\mathbb{Z}}(b),\tau^{\mathbb{Z}}(\sigma(b))}$$
which, if non-empty, will be fixed pointwise by $\tau$, and on which $\sigma$ will have a fixed point, contradicting the assumption that $G$ has no fixed points.
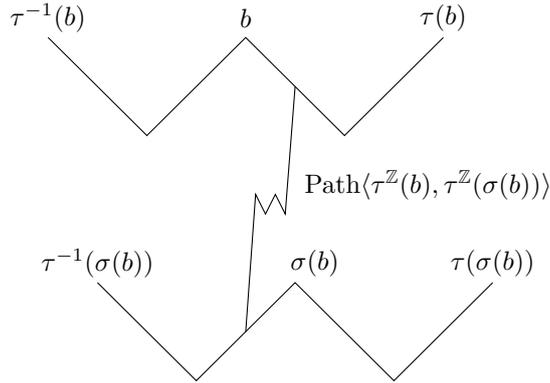
\begin{figure}[h!]\label{FigureNextSoon}
\begin{center}
\begin{tikzpicture}[scale=0.13]
\draw (0,30) -- (10,20) -- (20,30) -- (30,20) -- (40,30);
\draw (5,5) -- (15,-5) -- (25,5) -- (35,-5) -- (45,5);
\draw (25,25) -- (24,12) -- (23,14) -- (22,12) -- (21,14) -- (20,0);

\draw (0,32) node {$\tau^{-1}(b)$};
\draw (20,32) node {$b$};
\draw (40,32) node {$\tau(b)$};

\draw (5,7) node {$\tau^{-1}(\sigma(b))$};
\draw (27,7) node {$\sigma(b)$};
\draw (45,7) node {$\tau(\sigma(b))$};

\draw[anchor=west] (25,15) node {$\path{\tau^{\mathbb{Z}}(b),\tau^{\mathbb{Z}}(\sigma(b))}$};
\end{tikzpicture}
\end{center}
\caption{$\path{\tau^{\mathbb{Z}}(b),\tau^{\mathbb{Z}}(\sigma(b))}$}
\end{figure}

If $\path{\tau^{\mathbb{Z}}(b),\tau^{\mathbb{Z}}(\sigma(b))}$ is empty then we are in the situation depicted in Figure 17.
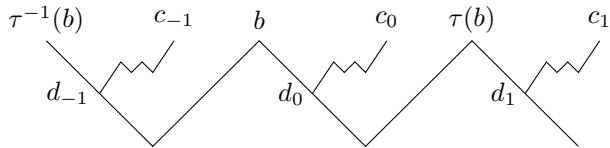
\begin{figure}[hb]\label{FigureNextSoon2}
\begin{center}
\begin{tikzpicture}[scale=0.14]
\draw (-20,5) -- (-10,-5) -- (0,5) -- (10,-5) -- (20,5) -- (30,-5);
\draw (-8,5) -- (-10,2) -- (-11,3) -- (-12,2) -- (-13,3) -- (-15,0);
\draw (12,5) -- (10,2) -- (9,3) -- (8,2) -- (7,3) -- (5,0);
\draw (32,5) -- (30,2) -- (29,3) -- (28,2) -- (27,3) -- (25,0);

\draw (-20,7) node {$\tau^{-1}(b)$};
\draw (0,7) node {$b$};
\draw (20,7) node {$\tau(b)$};

\draw (-8,7) node {$c_{-1}$};
\draw (12,7) node {$c_0$};
\draw (32,7) node {$c_1$};

\draw (-18,0) node {$d_{-1}$};
\draw (3,0) node {$d_0$};
\draw (23,0) node {$d_1$};
\end{tikzpicture}
\end{center}

\caption{$\path{\tau^{i}(b),c_i}$}

\end{figure}

In Figure 18 $c_0$ is $\sigma(b)$ and $c_{k}:=\tau^k(c_0)$, which forces $\sigma \tau^k (b)$ to be $c_j$ for some $j$ (whose relationship with $k$ will be deduced shortly).  Note that $\sigma$ and $\tau$ satisfy the identity $$\sigma \tau = \tau^{-1} \sigma$$ which implies the following equations:

$$\begin{array}{r c l}
c_i & = &  \sigma \tau^j(b) \\
 & = & \tau^{-j} \sigma (b) \\
 \tau^j (c_i) & = & c_0 \\
 c_{i+j} & = & c_0
\end{array}$$

so $c_i = \sigma(\tau^{-i}(b))$.  Let the $d_i$ be the points fixed by $\sigma$ on $\path{\tau^{i}(b),c_i}$ respectively.  Then $\bigcup \path{d_i,d_j}$ is a copy of Alt which is acted on as desired.

Let $\mathcal{A}$ be the family of copies of Alt in $M$.  We now show that if $\act(A,\aut(M))$ (the action of $\aut(M)$ on $A$) is isomorphic to $\dinfinity$ for some $A \in \mathcal{A}$ then $\act(A,\aut(M))$ cannot be contained in a \treecauser, thus showing the exclusivity of the theorem.

If for some $A \in \mathcal{A}$ the action of $\aut(M)$ is $\dinfinity$ then $\act(A,\aut(M)) \cong \dinfinity$ and there is no \treecauser contained in $\aut_{\lbrace A \rbrace}(M)$ that contains $\act(A,\aut(M))$.  Therefore if $\act(A,\aut(M))$ is contained in a \treecauser $X$, then
$$X \not\leq \aut_{\lbrace A \rbrace}(M)$$
In particular this implies that if $g \in X \setminus \act(A,\aut(M))$ then $g(A) \not= A$.

Let $A^U$ be the set of upper points of $A$, enumerated by $\lbrace \ldots, a_{-2}, a_0, a_2, \ldots \rbrace$.

Since $(A^U,\act(A,\aut(M)))$ is 1-transitive so is $(X(A^U),X)$, and so

$$(X(A^U),X) \cong (X(A^U), X_0)$$

where $X_0$ is one of the factors of $X$ (i.e. $\sym$, $\mathbb{Z} \wr \mathbb{Z}_2$ or $\prod S_n$).  This $X_0$ cannot be $\sym$ as then it would be possible to map the triple $(a_{-2}, a_0, a_2)$ to $(a_{-2}, a_2, a_0)$, but any map that does this has to change the length of $\path{ a_{-2},a_2 }$, and so cannot be an isomorphism.  This same argument prevents $X_0 \cong \prod S_n$.

Let $\sigma$ be the infinite order generator of $\act(A,\aut(M))$ and $\tau$ be the finite order generator.  Suppose $X_0 \cong \mathbb{Z} \wr \mathbb{Z}_2$, generated by $\alpha$, $\beta$ and $\gamma$, where $\alpha$ and $\beta$ have infinite order and $\gamma$ has finite order.  Since $\act(A,\aut(M))$ contains an element of finite order, both $\supp(\alpha)$ and $\supp(\beta)$ must have a non-empty intersection with $A^U$.

Since $\alpha$, $\beta$ and $\gamma$ generate $X$ and either preserve or switch $\supp(\alpha)$ and $\supp(\beta)$, every member of $\act(A,\aut(M))$ either preserves or switches $\supp(\alpha)$ and $\supp(\beta)$.  So both $\supp(\alpha) \cap A^U$ and $\supp(\beta) \cap A^U$ cannot both be singletons, as only the identity will preserve $\supp(\alpha) \cap A^U$ and $\supp(\beta) \cap A^U$ and no member of $\act(A,\act(M))$ will swap them.  Since $\supp(\alpha) \cap A^U$ is not a singleton, the action on it determines the action on the whole of $A^U$, and so $\alpha$ and $\beta$ cannot act independently.  So $X_0$ cannot be isomorphic to $\mathbb{Z} \wr \mathbb{Z}_2$.
\end{proof}

\begin{cor}
Let $M$ be a CFPO.  If there is an $A \subseteq M$ and a $G \leq \aut(M)$ such that:
\begin{enumerate}
\item $A$ is a copy of Alt;
\item $G \cong \dinfinity$; and
\item $G$ acts on $A$.
\end{enumerate}
then $M$ is not treelike.
\end{cor}
\begin{proof}
If $M$ is treelike then Theorem \ref{thm:SuperDinf} shows that $G$ is contained in a dendromorphic group, but Theorem \ref{thm:dinfinityinCFPOs} shows that this is impossible.
\end{proof}

\bibliography{bib}{}
\bibliographystyle{plain}
 
\end{document}